\documentclass[10pt]{article}
\usepackage{amssymb}
\usepackage{graphicx}
\usepackage{xcolor} 
\usepackage{tensor}
\usepackage{fullpage} 
\usepackage{amsmath}
\usepackage{amsthm}
\usepackage{verbatim}

\usepackage{enumitem}
\setlist[enumerate]{leftmargin=1.5em}
\setlist[itemize]{leftmargin=1.5em}

\usepackage{seqsplit}

\definecolor{green}{rgb}{0,0.8,0} 

\newtheorem{theorem}{Theorem}[section]
\newtheorem{corollary}[theorem]{Corollary}
\newtheorem{lemma}[theorem]{Lemma}
\newtheorem{proposition}[theorem]{Proposition}

\newtheorem{innercustomthm}{Theorem}
\newenvironment{customthm}[1]
{\renewcommand\theinnercustomthm{#1}\innercustomthm}
{\endinnercustomthm}
\theoremstyle{definition}
\newtheorem{definition}[theorem]{Definition}

\theoremstyle{remark}
\newtheorem{remark}[theorem]{Remark}

\numberwithin{equation}{section}
\newcommand{\nrm}[1]{\Vert#1\Vert}

\newcommand{\abs}[1]{\vert#1\vert}

\newcommand{\nnrm}[1]{{\vert\kern-0.25ex\vert\kern-0.25ex\vert #1 
    \vert\kern-0.25ex\vert\kern-0.25ex\vert}}

\newcommand{\lap}{\Delta}

\newcommand{\rd}{\partial}
\newcommand{\nb}{\nabla}

\newcommand{\alp}{\alpha}

\newcommand{\gmm}{\gamma}

\newcommand{\dlt}{\delta}

\newcommand{\eps}{\epsilon}

\newcommand{\omg}{\omega}
\newcommand{\Omg}{\Omega}

\newcommand{\bfc}{{\bf c}}

\newcommand{\bff}{{\bf f}}

\newcommand{\bfn}{{\bf n}}

\newcommand{\bfp}{{\bf p}}

\newcommand{\bfu}{{\bf u}}

\newcommand{\bfrho}{{\boldsymbol{\rho}}}

\newcommand{\bbN}{\mathbb N}

\newcommand{\bbR}{\mathbb R}

\newcommand{\bbT}{\mathbb T}

\newcommand{\calO}{\mathcal O}

\newcommand{\calX}{\mathcal X}

\newcommand{\bfomg}{\boldsymbol{\omg}}		

\newcommand{\tu}{\tilde{\bfu}}					

\newcommand{\trho}{\tilde{\bfrho}}
\newcommand{\tc}{\tilde{\bfc}}
\newcommand{\tf}{\tilde{\bff}}  

\begin{document}

\title{Well-posedness and singularity formation for inviscid Keller--{Segel}--fluid system of consumption type}
\author{In-Jee Jeong\thanks{Department of Mathematical Sciences and RIM, Seoul National University. E-mail: injee\_j@snu.ac.kr}\and Kyungkeun Kang\thanks{Department of Mathematics, Yonsei University. E-mail: kkang@yonsei.ac.kr}}
\date{\today}

 


\maketitle


\begin{abstract}
	We consider the Keller--Segel system of consumption type coupled with an incompressible fluid equation. The system describes the dynamics of oxygen and bacteria densities evolving within a fluid. We establish local well-posedness of the system in Sobolev spaces for partially inviscid and fully inviscid cases. In the latter, additional assumptions on the initial data are required when either the oxygen or bacteria density touches zero. Even though the oxygen density satisfies a maximum principle due to consumption, we prove finite time blow-up of its $C^{2}$--norm with certain initial data. 
\end{abstract}


\section{Introduction}

\subsection{The systems}

In this paper, we study a mathematical model, so called Keller--Segel-fluid system, formulating the dynamics of oxygen, swimming bacteria, and viscous incompressible fluids in two or three dimensions.
Such a model was introduced by Tuval et al. \cite{Tuval2277}, proposing the
dynamical behaviors of bacteria, Bacillus subtilis, living in water and consuming oxygen.
To be more precise, we consider the following Keller--Segel-fluid (KSF) system, in its general form, which describes the evolution of the bacteria density $\bfrho$, the oxygen density $\bfc$, and the fluid velocity $\bfu$:  

\begin{equation} \tag{KSF} \label{eq:KSF}
\left\{
\begin{aligned}
&\rd_t \boldsymbol{\rho}  + \bfu\cdot\nabla \bfrho = D_\rho \lap \bfrho  -\nabla \cdot (\chi(\bfc)\bfrho\nabla \bfc) , \\
&\rd_t \bfc + \bfu\cdot\nabla \bfc = D_c\lap \bfc - k(\bfc)\bfrho  , \\
&\rd_t \bfu + \bfu\cdot\nabla \bfu + \nabla \bfp = D_u \lap \bfu + \bfrho\nabla \phi , \\
& \mathrm{div}\, \bfu = 0. 
\end{aligned}
\right.
\end{equation} 
 Here $\chi(\cdot)$ and $k(\cdot)$ are non-negative and smooth functions denoting chemotactic sensitivity and the oxygen consumption rate, respectively. The main features of the model are that the bacteria moves towards the region of high oxygen concentration and the oxygen is being consumed by the bacteria. Moreover, both the oxygen and bacteria concentrations are being transported by the fluid velocity, and the bacteria density affects the fluid velocity via some potential function $\phi$
(see \cite{Tuval2277} for more details on biological phenomena and \cite{CFKLM} for numerical investigations). 
Together with \eqref{eq:KSF}, we shall consider the simplified model obtained by  formally taking  $\bfu \equiv 0$:
\begin{equation} \tag{KS} \label{eq:KS}
\left\{
\begin{aligned}
& \rd_t\bfrho = D_\rho\lap\bfrho - \nabla\cdot(\chi(\bfc)\bfrho\nabla \bfc) ,\\
& \rd_t \bfc = D_c\lap \bfc - k(\bfc)\bfrho ,
\end{aligned}
\right.
\end{equation} which is  a Keller--Segel-type system (see e.g. \cite{MR2860628}). We shall mainly consider the above systems in the domain $\bbT^d$ for some $d\ge 1$, but many of our results carry over to the case of $\bbR^k\times\bbT^{d-k}$ for $0<k\le d$ and bounded domains of $\bbR^d$ with smooth boundary. In this note, we always assume that $\bfrho, \bfc \ge 0$, which is preserved by the dynamics of \eqref{eq:KSF} and \eqref{eq:KS}. Observe that the assumption $k \ge 0$ yields a maximum principle for $\bfc$: as long as the solution exists, $\nrm{\bfc(t)}_{L^\infty}$ decreases with time.

We can compare \eqref{eq:KS} to the classical Keller--Segel (cKS)  model  
suggested by Patlak \cite{MR81586} and Keller--Segel \cite{MR3925816} to describe the motion of amoeba, Dictyostelium discoideum living in soils, which is given as 
\begin{equation} \tag{cKS} \label{eq:cKS}
\left\{
\begin{aligned}
& \rd_t\bfrho = D_\rho\lap\bfrho - \nabla\cdot(\chi(\bfc)\bfrho\nabla \bfc) ,\\
& \tau \rd_t \bfc = D_c\lap \bfc +\alpha\bfrho -\beta \bfc,
\end{aligned}
\right.
\end{equation} 
where $\alpha$ and $\beta$ are positive constants indicating production and decaying rate of chemical 
substance, respectively.
Here the system \eqref{eq:cKS} consists of biological cell density $\bfrho$ and the concentration of chemical attractant
substance $\bfc$ that is produced by $\bfrho$. Compared to \eqref{eq:KS}, the second equation of \eqref{eq:cKS} has a positive coefficient for $\rho$, which seems to cause significantly different behaviors of solutions.
In general, finite time blow-up results of the system \eqref{eq:cKS} have been known for large data in case that $D_\rho>0$, $D_c>0$ and $\tau\ge 0$ (see e.g. \cite{MR1046835}, \cite{MR1627338}, \cite{MR3115832}). There 
are also some related results for inviscid models with either $D_\rho=0$ or $D_c=0$ in \eqref{eq:cKS}, but we are not going to list all of the related results here (see e.g.  \cite{MR2099126}, \cite{MR3223867},  \cite{MR3473109}, reference therein).

\subsection{Main results}

Much of the mathematical literature concerning the systems \eqref{eq:KSF} and \eqref{eq:KS} was devoted to establishing global well-posedness of solutions, assuming that the dissipation coefficients $D_\rho$, $D_c$, and $D_u$ are strictly positive. 
Since our main concerns are focused on partially or fully inviscid case, we shall not try to review all related results, but instead, 
just list  a few of them related to the issue of global well-posedness.
Ever since mathematical results were made initially in  \cite{MR2659745} and \cite{MR2754058} regarding local well-posedness and stability near constant steady state of solutions, the issue on global well-posedness or blow-up in a finite time has received lots of speculation. Among other results, it was shown in \cite{MR2876834} that regular solutions exist globally in time for general data in two dimensions under certain assumptions on $\chi$ and $k$ (we refer to \cite{MR3149063} {and \cite{MR3605965}} for asymptotic behaviors of solutions). We note that the paper \cite{CKL1} establishes global well-posedness of smooth solutions with different conditions on $\chi$ and $k$ in $\mathbb R^2$ (see also \cite{CKL2}, \cite{MR3450942} and \cite{MR3755872} for 
blow-up criteria, temporal decay and  asymptotic of solutions, { and \cite{MR3542616} for weak solutions}).
Very recently, \cite{MR4185378} obtained global well-posedness  with slightly relaxed the conditions on $\chi$ and $k$, compared to those in \cite{MR2876834}.

On the other hand, the goal in this paper is to study local well-posedness of smooth solutions when some or all of the dissipation coefficients vanish. To the best of the authors' knowledge, the only result in this direction is the one given in \cite{CKL2}, where the authors prove global well-posedness in the case $D_c = 0$ and $D_\rho>0$, provided that $L^{N/2}$-norm of $\rho_0$ is sufficiently small. 

Our primary motivation for studying the inviscid (or partially inviscid) system is that norm growth for the inviscid system could be responsible for instabilities arising in the slightly viscous case. Indeed, in the original paper \cite{Tuval2277} (see also \cite{CFKLM}) where the system was introduced and instabilities were explored numerically, the dissipation coefficients take values in $10^{-4} \sim 10^{-2}$ after nondimensionalization. 

In the presence of local-in-time wellposedness for the inviscid system, one can establish in general that the corresponding viscous problem for the same initial data cannot blow-up \textit{earlier} than the inviscid case; this type of result goes back to a classical work of Constantin \cite{Con2}. On the contrary, if the inviscid equation blows up in finite time, it can be expected that the slightly viscous system should go through a large norm growth, which is divergent as the viscosity constant goes to zero. Similar phenomena should occur in the extreme case of the inviscid system being \textit{ill-posed}, meaning that the Sobolev norm of the solution could blow up in an arbitrary short interval of time. In other words, ``norm inflation'' for the inviscid equation could be responsible for instabilities appearing in the corresponding slightly viscous case. This type of result was recently established for the incompressible Euler and Navier--Stokes equations in \cite{JY, JY2} { (see also  \cite{MR3223867}, \cite{MR3782543}, \cite{CKL2}, and \cite{MR3473109} for the case of partially invicid case of Keller--Segel type equations)}.

The conclusion of this paper is that  the systems \eqref{eq:KSF} and \eqref{eq:KS} are indeed \textit{locally well-posed} even in the fully inviscid case $D_\rho=D_c=D_u=0$. This is somewhat surprising, since the equation for $\bfrho$ involves two derivatives of $\bfc$. Roughly speaking, loss of one derivative can be gained by pairing the derivative of $\bfc$ together with $\bfrho$, and the other loss of derivative can be handled by performing a \textit{modified} energy estimate: the key observation is that in the expression of \begin{equation*}
\begin{split}
\frac{d}{dt}\int \frac{k(\bfc )}{\chi(\bfc)} (\rd^m\bfrho)^2 + \bfrho (\rd^{m+1}\bfc)^2,
\end{split}
\end{equation*} the top order terms completely cancel each other, which allows to close an a priori estimate. Here, the signs $\chi, k\ge 0$ are fundamental: if one of the either sign is reversed, then  we expect the resulting systems to be strongly \textit{ill-posed} in sufficiently regular Sobolev spaces. Such an ill-posed behavior could be easily seen from a simple linear analysis which is performed later in Section \ref{sec:inviscid}, Remark \ref{rem:illposed}. This already demonstrates a striking difference between Keller--Segel equations of consumption and production type. {Furthermore, we remark that the modified energy method is robust enough to handle the general flux case (considered very recently in \cite{WinklerIMRN}) where $\nb\cdot (\chi(\bfc)\bfrho \nb\bfc)$ in (KS) is replaced by $\nb\cdot (\chi(\bfc)\bfrho  S \nb\bfc)$, under some assumptions on the matrix $S$. See Remark \ref{rem:rotation} below. 
}

\subsubsection{Dissipation only on the oxygen concentration}

Our first main results concerns local well-posedness of \eqref{eq:KSF} and \eqref{eq:KS} with smooth initial data when $D_u = 0$ and $D_\rho=0$ but $D_c>0$. While the precise statement is given in Theorem \ref{thm:c-viscous}, we state the main result briefly here, in the simplest setup of $\bbT^d$. {We consider $d\ge2$, simply because in the case $d=1$, there are no non-trivial incompressible velocity fields.}

\begin{customthm}{A}[Local well-posedness with dissipation only on $\bfc$]\label{MainThm1}
	The system \eqref{eq:KSF} in the case $D_{\rho}, D_u \ge 0$ and $D_c >0$ is locally well-posed in $(\bfrho,\bfc,\bfu) \in (H^m\times H^{m+1}\times H^{m})(\bbT^d)$ for any $m>\frac{d}{2}+1$ {with $d\ge2$}.  
\end{customthm}

The proof of Theorem \ref{MainThm1} is rather straightforward, the main observation being that the $H^m$-regularity of $\bfrho$ and $\bfu$ should be paired with the $H^{m+1}$-regularity of $\bfc$. In Theorem \ref{thm:c-viscous}, we actually handle the case of bounded domains with Neumann boundary condition for $\bfc$, which gives several technical complications. We defer the detailed discussion to Section \ref{sec:c-viscous}, where we state and prove a precise version of the above result.

\subsubsection{Fully inviscid case}

As we shall demonstrate in Section \ref{sec:inviscid}, a simple linear analysis around trivial steady states of \eqref{eq:KSF} shows that the system could be well-posed even in the absence of dissipation. Let us first consider the simpler case of \eqref{eq:KS}. The key ingredient in the proof of local well-posedness, as illustrated in the above, is to use the modified norms $\nrm{\bfc^{\frac{1}{2}}\rd^m\bfrho}_{L^2}$ and $\nrm{\bfrho^{\frac{1}{2}}\rd^{m+1}\bfc}_{L^2}$, in the simplest model case of $k(z)=z$ and $\chi(z)=1$. Combining these modified norms, we can guarantee a crucial cancellation of top order terms, which allows one to close an a priori estimate in $(\bfrho,\bfc)\in H^{m}\times H^{m+1}$ as long as $\bfrho$ and $\bfc$ stay away from zero. With such a modified energy estimate for \eqref{eq:KS}, local well-posedness can be extended without much difficulty to the case of \eqref{eq:KSF}, when the fluid viscosity is strictly positive; that is, $D_u>0$. The fully inviscid assumption $D_u = 0$ gives rise to a serious problem; to explain it in simple terms, the equation for $\rd^{m+1}\bfc$ involves $\rd^{m+1}\bfu$, and in turn, the equation for $\rd^{m+1}\bfu$ involves $\rd^{m+1}\bfrho$. Finally, the equation for $\rd^{m+1}\bfrho$ costs $m+3$ derivatives of $\bfc$. The way out of this loop is to consider modified energy estimate for \textit{good variables}; it turns out that the time derivative of the quantity \begin{equation}\label{eq:key}
\begin{split}
\int \frac{k(\bfc)}{\chi(\bfc)} \left| \rd^m\bfrho + \frac{\nb\bfc }{k(\bfc)} \cdot \rd^m\bfu \right|^2  + \bfrho \left| \rd^m \nb \bfc - \frac{\nb\bfc}{\bfrho k(\bfc)} \nb \bfc \cdot \rd^m \bfu \right|^2 + \left| \rd^m \bfu \right|^2 
\end{split}
\end{equation} satisfies a good energy structure, for any $m$. Note that when $\bfrho$ and $\bfc$ are strictly positive, the above is equivalent with the usual $\dot{H}^m$-norm for $(\bfrho,\nb\bfc,\bfu)$. This result is briefly stated as follows:

\begin{customthm}{B}[Local well-posedness in the inviscid case with non-vanishing $\bfrho$ and $\bfc$]\label{MainThm2}
	The system \eqref{eq:KSF} in the case $D_{\rho}=D_u=D_c = 0$ is locally well-posed in $(\bfrho,\bfc,\bfu) \in (H^m\times H^{m+1}\times H^{m})(\bbT^d)$ for any $m>\frac{d}{2}+1$, {$d\ge2$}, and strictly positive $\bfrho_0$ and $\bfc_0$.  
\end{customthm}

The precise statement of the above is given as Theorem \ref{thm:lwp-away} below. Given this result, it is natural to ask what happens when either $\bfrho$ or $\bfc$ vanishes at some points. In this case, the inclusion of velocity (with $D_u=0$) gives rise to a serious issue as the coefficients \begin{equation*}
\begin{split}
\frac{\nb\bfc }{k(\bfc)}, \quad \frac{\nb\bfc}{\bfrho k(\bfc)} \nb \bfc
\end{split}
\end{equation*} are in general \textit{singular} (unbounded when $\bfc$ vanishes), which does not allow one to deduce regularity of $\bfrho$ and $\bfc$ even when the integral \eqref{eq:key} is finite. Restricting ourselves to the case of \eqref{eq:KS}, we can still ensure existence and uniqueness of smooth solutions, upon imposing some assumptions on the initial data (which we prove to propagate in time) near their zeroes. For now, let us roughly state it as follows:
\begin{customthm}{C}[Local well-posedness in the inviscid case with possibly vanishing $\bfrho$ and $\bfc$]\label{MainThm3}
	{For any $d\ge1$,} the system \eqref{eq:KS} in the case $D_{\rho}=D_c = 0$ is locally well-posed in sufficiently regular Sobolev spaces with possibly vanishing initial data $\bfrho_0$ and $\bfc_0$, under some additional assumptions on the data near their zeroes. 
	
	Within this locally well-posed class of solutions, there exist initial data $\bfrho_0, \bfc_0 \in C^\infty(\bbT^d)$ whose unique solution to \eqref{eq:KS} blows up in finite time; more concretely, there exists some $T^*>0$ such that $\nrm{\bfc(t)}_{C^2} \rightarrow+\infty$ as $t \to T^*$.  
\end{customthm}

{

\begin{remark}
	The terminology ``sufficiently regular'' and the further assumptions that we need to impose in the case of vanishing initial data are somewhat complicated to define and therefore we defer the precise statements to Theorems \ref{thm:lwp-rho-away} and \ref{thm:quadratic}. 
\end{remark}

}

As stated in this theorem, we are able to prove singularity formation within this class, which is surprising since $\bfc$ satisfies a maximum principle: $\nrm{\bfc}_{L^\infty}$ is decreasing with time. A rough heuristics for this singularity formation can be given as follows: arrange initial data for \eqref{eq:KS} so that $\bfrho$ takes its minimum value where $\bfc$ reaches its maximum. Then, since the consumption of $\bfc$ is proportional to $\bfrho$, the profile of $\bfc$ becomes steeper near its maximum. This steepening of $\bfc$ is enhanced by the dynamics of $\bfrho$, which concentrates towards the maximum of $\bfc$. It will be very interesting to extend this singularity formation to the case of a partially viscous \eqref{eq:KSF} system. 
When vanishing of either $\bfc$ or $\bfrho$ occurs, it is no longer true that having a bound on the quantity \eqref{eq:key} ensures smoothness. Indeed, the Keller--Segel model in this case should be interpreted as a \textit{degenerate hyperbolic} system, and our main technical tool for local well-posedness is the following weighted Gagliardo--Nirenberg--Sobolev inequalities: \begin{lemma}[Key Lemma]  \label{lem:key}
		 Let $g \ge 0$ on $\bbT^d$. For any $\gmm>0$, $m \in \bbN$, and a $k$-tuple $(\ell_1, \cdots, \ell_k)$ of $d$-vectors satisfying $m = |\ell_1| + \cdots + |\ell_k|$, we have that\begin{equation}\label{eq:key-1}
			\begin{split}
				\int \frac{ 1 }{  g^{2k-2+\gmm} } \prod_{1\le i \le k} |\rd^{\ell_i} g|^2 \lesssim_{\gmm,m,d} \int \frac{|\nb^m g|^2}{ g^{\gmm}} + \int \frac{|\nb g|^{2m}}{ g^{2m-2+\gmm}},
			\end{split}
		\end{equation} assuming that the right hand side is finite. In particular, under the same assumption, we obtain that for any $0\le n \le m$, 
		\begin{equation} 
			\begin{split}
				\nrm{ \nb^{m-n}( \frac{\nb^{n} g}{ { g}^{\frac{\gmm}{2}} })}_{L^2}^2 \lesssim_{\gmm,m,d} \int \frac{|\nb^m g|^2}{ g^{\gmm}} + \int \frac{|\nb g|^{2m}}{ g^{2m-2+\gmm}}
			\end{split}
		\end{equation} and  for any $k\ge1$  
		\begin{equation} 
			\begin{split}
				\sup_{x \in \bbT^d} \frac{|\nb g(x)|^{k}}{ g(x)^{k-1+\frac{\gmm}{2}}} \lesssim_{\gmm,k,d} \sum_{m=1}^{\lfloor \frac{d}{2} \rfloor +1 + k} \left( \int \frac{|\nb^m g|^2}{ g^{\gmm}} + \int \frac{|\nb g|^{2m}}{ g^{2m-2+\gmm}}  \right) .
			\end{split}
		\end{equation} 
	\end{lemma} While there exist several versions of Gagliardo--Nirenberg--Sobolev inequalities with a spatial weight, we are not aware of ones in which the weight depends on the function itself. Note that if $g$ is $C^\infty$--smooth and vanishes with sufficiently high order near its zeroes, depending on given $m,\gmm$, the right hand side of \eqref{eq:key-1} is finite. Moreover, for bounded $g$, the right hand side of \eqref{eq:key-1} trivially controls the usual Sobolev norm $H^{m}$. We believe that the above set of inequalities will be a useful  tool in the analysis of \textit{degenerate hyperbolic} systems. Let us defer further technical discussion relevant to the statements and difficulties involved with Theorem \ref{MainThm3} to the introduction of Section \ref{sec:inviscid}. For now, let us observe that unlike the other results, Theorem \ref{MainThm3} is very sensitive to the profiles of $\chi$ and $k$ (in the region where $\bfc$ is small). Under suitable assumptions on $\chi$ and $k$, this result could be extended to the case of \eqref{eq:KSF}, but we shall not pursue this generalization in this work.

\subsubsection*{Notation and conventions}

Given a function $f : \bbR^d \rightarrow \bbR $ and an integer $m \ge 0$, $\nabla^mf$ denotes the $d^m$-dimensional vector consisting of all possible partial derivatives of $f$ of order $m$. We define the $H^m$-spaces as 
\begin{equation*} 
\begin{split}
& \nrm{f}_{H^m}^2 = \sum_{j = 0}^m \nrm{f}_{\dot{H}^j}^2 =  \sum_{j = 0}^m \nrm{\nabla^j f}_{L^2}^2
\end{split}
\end{equation*} Next, we define the $L^p$-norm by \begin{equation*} 
\begin{split}
& \nrm{f}_{L^p}^p = \int |f|^p \, dx, \quad \nrm{f}_{L^\infty} = \mathrm{esssup}|f(x)|. 
\end{split}
\end{equation*} 


\section{Local well-posedness with dissipation only on $\bfc$}\label{sec:c-viscous}

In this section, we prove local well-posedness for smooth solutions of \eqref{eq:KSF} in the case when there is dissipation only in $\bfc$. We shall assume that the equation is posed on a bounded domain $\Omg \subset \bbR^d$ with $C^\infty$-smooth boundary. The proof we present easily adapts to the case of domains without boundary, for example to $\bbR^k\times\bbT^{d-k}$. Presence of the boundary makes the well-posedness proof more technically involved since in general the derivatives of the dependent variables do not satisfy simple boundary conditions. 

Without loss of generality, we may normalize $D_c=1$, and the system becomes 
 \begin{equation}  \label{eq:KSF-bdd}
\left\{
\begin{aligned}
&\rd_t \boldsymbol{\rho}  + \bfu\cdot\nabla \bfrho = -\nabla \cdot ( \bfrho\nabla \bfc) , \\
&\rd_t \bfc + \bfu\cdot\nabla \bfc =  \lap \bfc - \bfc\bfrho  , \\
&\rd_t \bfu + \bfu\cdot\nabla \bfu + \nabla \bfp =  \bfrho\nabla \phi , \\
& \mathrm{div}\, \bfu = 0. 
\end{aligned}
\right.
\end{equation} Moreover, we shall impose the natural boundary conditions \begin{equation}\label{eq:KSF-bdry}
\begin{split}
\bfn \cdot \bfu = 0, \quad \rd_{\bfn} \bfc = 0  \quad \mbox{on}\quad \partial\Omg 
\end{split}
\end{equation} where $\bfn$ is the outward unit normal vector on $\partial\Omg$. This ensures that the characteristics defined by $\bfu$ and $\bfu + \nb\bfc$ preserves $\Omg$. 

\begin{theorem}\label{thm:c-viscous}
	Assume that the initial triple satisfies $$(\bfrho_0,\bfc_0,\bfu_0) \in (H^{m}\times H^{m+1}\times H^m)(\Omg)$$ and the boundary conditions $$\mathrm{div}(\bfu_0) = 0, \quad \mbox{in}\quad \Omg,$$ $$ \bfn \cdot \bfu_0 = 0,\quad \mbox{on} \quad \partial\Omg,$$ 
		and 
		$$\rd_{\bfn}(\rd_t^j \bfc)(t=0) = 0,\quad 0\le j \le \left\lceil \frac{m+1}{2} \right\rceil \quad \mbox{on}\quad \partial\Omg.$$
	Then, there exist $T>0$ and a unique solution satisfying 
	$$(\bfrho,\bfc,\bfu) \in L^\infty_t([0,T);(H^{m}\times H^{m+1}\times H^m)(\Omg)), \qquad \bfc \in L^2_t([0,T);H^{m+2}(\Omg)) $$ 
	of \eqref{eq:KSF-bdd}--\eqref{eq:KSF-bdry} satisfying $(\bfrho,\bfc,\bfu)(t=0) = (\bfrho_0,\bfc_0,\bfu_0)$. The unique solution blows up at some $0<T^* < + \infty$ if and only if \begin{equation*}
		\begin{split}
			\lim_{t\nearrow T^*} \left( \nrm{\bfc(t)}_{W^{2,\infty}} + \nrm{\bfrho(t)}_{W^{1,\infty}} + \nrm{\bfu(t)}_{W^{1,\infty}} \right) = +\infty. 
		\end{split}
	\end{equation*} 
\end{theorem}

\begin{remark}
	The assumption $\rd_{\bfn}(\rd_t^j \bfc)(t=0) = 0$ is a compatibility condition for the solution to be uniformly smooth up to the boundary, which is easily guaranteed when $\bfc_{0}$ is supported away from the boundary. Using the equation for $\bfc$, one can rewrite this assumption in terms of $(\bfrho_0,\bfc_0,\bfu_0)$ and their derivatives. {It is already complicated in the case $j=1$: we have \begin{equation*}
			\begin{split}
				\rd_{\bfn} (\rd_t \bfc)(t=0)= 	\rd_{\bfn}( -  \bfu_0\cdot\nabla \bfc_0  + \lap \bfc_0 - \bfc_0\bfrho_0).  
			\end{split}
	\end{equation*} }
\end{remark}

\begin{remark}
{Note that we have taken $\chi=1$ and $k(z)=z$. In the case of Theorem \ref{thm:c-viscous}, extending local well-posedness to the case of general smooth $\chi$ and $k$ is straightforward as additional terms arising from derivatives of $\chi$ and $k$ are not the highest order terms.} 
\end{remark}

\begin{proof}
	For simplicity, we shall fix $d = 2$ and $m = 3$. There are no essential differences in handling the general case. {To begin with, let us obtain \textit{a priori estimates} for a solution $(\bfrho,\bfc,\bfu)$ of \eqref{eq:KSF-bdd}, which is assumed to be sufficiently smooth and satisfy the boundary conditions, so that the following computations can be justified. Later we shall show how such estimates can be justified, using an approximate solution sequence. We shall divide the proof of a priori estimates into three parts, which correspond to $L^2$, $H^2$, and $H^4$ estimate in terms of $\bfc$. 
} 
	
	\medskip
	
	\noindent \textit{(i) $L^2$ estimate}: To begin with, directly from \eqref{eq:KSF-bdd} we obtain the following $L^2$-estimate: \begin{equation}\label{eq:KSF-bdd-L2}
	\begin{split}
	\frac{1}{2}\frac{d}{dt}(\nrm{\bfrho}_{L^2}^2 + \nrm{\bfc}_{L^2}^2+\nrm{\bfu}_{L^2}^2) + \nrm{\nb\bfc}_{L^2}^2 \lesssim \nrm{\lap\bfc}_{L^\infty}\nrm{\bfrho}_{L^2}^2 + \nrm{\bfrho}_{L^\infty} \nrm{\bfc}_{L^2}^2 + \nrm{\bfrho}_{L^2}\nrm{\bfu}_{L^2}. 
	\end{split}
	\end{equation} We have used the boundary conditions for $\bfu$ and $\bfc$. 
	
	\medskip
	
	\noindent \textit{(ii) $H^2$ estimate}: Next, we consider the set of variables $(\nb\bfrho, \bfc_t,{\nb\bfu})$ where $\bfc_t:=\rd_t\bfc$: \begin{equation}\label{eq:KSF-bdd-der}
	\left\{ \begin{aligned}
	&\rd_t\nb\bfrho + (\bfu+\nb\bfc)\cdot\nb (\nb\bfrho) = -\nb\bfrho \lap\bfc - \bfrho\cdot\nb\lap\bfc - (\nb\bfu+\nb^2\bfc)\cdot\nb\bfrho  \\
	&\rd_t\bfc_t + \bfu\cdot\nb\bfc_t = \lap\bfc_t - \bfrho\bfc_t - \bfc \bfrho_t - \bfu_t\cdot\nb\bfc  \\
	&{\rd_t\nb  \bfu+ \bfu\cdot\nb \nb  \bfu= - \nb\bfu \cdot\nb\bfu - \nb^2 \bfp + \nb(\bfrho\nb\phi) .}
	\end{aligned} \right. 
	\end{equation}  
	Here, we are using $\bfrho_t = \rd_t\bfrho$ and $\bfu_t = \rd_t\bfu$. It is important that (formally) $\bfc_t$ satisfies the same boundary conditions with $\bfc$. In the equation for $\nb\bfrho$, using that $\bfn\cdot (\bfu+\nb\bfc) = 0$ on $\partial\Omg$, we obtain that \begin{equation*}
	\begin{split}
	\frac{d}{dt} \nrm{\nb\bfrho}_{L^2}^2 \lesssim \nrm{{\nb^2}\bfc}_{L^\infty}\nrm{\nb\bfrho}_{L^2}^2 + \nrm{\bfrho}_{L^\infty}\nrm{\nb\lap\bfc}_{L^2}\nrm{\nb\bfrho}_{L^2} + \nrm{\nb\bfu}_{L^\infty}\nrm{\nb\bfrho}_{L^2}^2 . 
	\end{split}
	\end{equation*} Then,\begin{equation}\label{eq:rho-H2}
	\begin{split}
	\frac{d}{dt} \nrm{\nb\bfrho}_{L^2}^2 \le \eps\nrm{\nb\lap\bfc}_{L^2}^2 + C_\eps (\nrm{\bfrho}_{L^\infty}^2 + \nrm{\lap\bfc}_{L^\infty}+\nrm{\nb\bfu}_{L^\infty} )\nrm{\nb\bfrho}_{L^2}^2  
	\end{split}
	\end{equation}	for any small $\eps>0$. We now move on to the estimate of $\bfc_t$.
	First, we note that 
	\begin{equation*}
	\begin{split}
	\nrm{\bfu_t}_{L^2} \lesssim \nrm{\bfu\cdot\nb\bfu}_{L^2}+\nrm{\nb\bfp}_{L^2} + \nrm{\bfrho\nb\phi}_{L^2} \lesssim \nrm{\nb\bfu}_{L^\infty} \nrm{\bfu}_{L^2} + \nrm{\bfrho}_{L^2}.
	\end{split}
	\end{equation*} To estimate the pressure term appearing in the right hand side, recall that $\bfp$ is the unique solution (up to a constant) to \begin{equation*}
	\begin{split}
	\lap \bfp & = - \sum_{1 \le i,j \le 2} \rd_i \bfu^j \rd_j \bfu^i + \nb\cdot(\bfrho\nb\phi) \mbox{ in }  \Omg, \\
	\rd_{\bfn}\bfp & = \bfrho \rd_{\bfn} \phi  \mbox{ on } \partial\Omg.
	\end{split}
	\end{equation*} Therefore, $\bfp$ satisfies the estimate \begin{equation}\label{eq:pressure}
	\begin{split}
	\nrm{\nb \bfp}_{H^m} \lesssim \nrm{\nb\bfu}_{L^\infty}\nrm{\bfu}_{H^m} + \nrm{\nb\phi}_{H^m} \nrm{\bfrho}_{H^m}
	\end{split}
	\end{equation} for each $m\ge0$ (see \cite{GiTr} as well as Step 1 of the proof of Theorem \ref{thm:lwp-away} below). Since we may assume that  $\int_\Omg \bfp = 0$, we have from Poincar\'{e} inequality that $\nrm{\bfp}_{L^2} \lesssim \nrm{\nb\bfp}_{L^2}$. 

Next, with an integration by parts, \begin{equation*}
	\begin{split}
	\int \bfc \bfrho_t \bfc_t = -\int \bfc \bfu\cdot\nb\bfrho \bfc_t + \int \bfrho\nb\bfc \cdot \nb (\bfc\bfc_t). 
	\end{split}
	\end{equation*} This gives the bound \begin{equation*}
	\begin{split}
	\left| \int \bfc \bfrho_t \bfc_t  \right| \lesssim \nrm{\bfc}_{L^\infty}\nrm{\bfu}_{L^\infty} \nrm{\nb\bfrho}_{L^2}\nrm{\bfc_t}_{L^2} + \nrm{\nb\bfc}_{L^\infty}\nrm{\bfrho}_{L^\infty}   \nrm{\nb\bfc}_{L^2}\nrm{\bfc_t}_{L^2}+ \nrm{\bfrho}_{L^\infty}   \nrm{\bfc}_{L^\infty}\nrm{\nb\bfc}_{L^2}\nrm{\nb\bfc_t}_{L^2}  . 
	\end{split}
	\end{equation*}
	Moreover, using the boundary condition for $\bfc_t$, we can use integration by parts in the convection term to obtain \begin{equation*}
	\begin{split}
	\frac{1}{2}\frac{d}{dt}\nrm{\bfc_t}_{L^2}^2 + \nrm{\nb\bfc_t}_{L^2}^2 &\le C \left(\nrm{\bfrho}_{L^\infty}\nrm{\bfc_t}_{L^2}^2 + \nrm{\bfc}_{L^\infty}\nrm{\bfu}_{L^\infty} \nrm{\bfc_t}_{L^2}\nrm{\nb\bfrho}_{L^2} + \nrm{\nb\bfc}_{L^\infty}\nrm{\bfrho}_{L^\infty} \nrm{\nb\bfc}_{L^2}\nrm{\bfc_t}_{L^2} \right) \\
	&\quad {+ C\nrm{\bfu}_{L^\infty}\nrm{\nb \bfu}_{L^\infty} \nrm{\nb\bfc}_{L^2} \nrm{\bfc_t}_{L^2} } + \eps \nrm{\nb\bfc_t}_{L^2}^2 + C_\eps \nrm{\bfc}_{L^\infty}^2 \nrm{\bfrho}_{L^\infty}^2 \nrm{\nb\bfc}_{L^2}^2 . 
	\end{split}
	\end{equation*} From the equation for $\bfc$, we have \begin{equation*}
	\begin{split}
	\nrm{\lap\bfc}_{L^2} & \lesssim \nrm{\bfc_t}_{L^2} + \nrm{\bfc\bfrho}_{L^2} + \nrm{\bfu\cdot\nb\bfc}_{L^2}  \lesssim  \nrm{\bfc_t}_{L^2} + \nrm{\bfrho}_{L^\infty} \nrm{\bfc}_{L^2} + \nrm{\bfu}_{L^\infty}\nrm{\bfc}_{L^2}^{\frac{1}{2}} \nrm{\bfc}_{H^2}^{\frac{1}{2}} 
	\end{split}
	\end{equation*} which gives \begin{equation*}
	\begin{split}
	\nrm{\bfc}_{H^2} & \lesssim \nrm{\bfc_t}_{L^2} + \left(1 +  \nrm{\bfrho}_{L^\infty} + \nrm{\bfu}_{L^\infty}^2 \right) \nrm{\bfc}_{L^2} 
	\end{split}
	\end{equation*} after using $\epsilon$-Young inequality. Similarly, we can estimate \begin{equation}\label{eq:rho-H2-prime}
	\begin{split}
	\nrm{\nb\lap\bfc}_{L^2} &\lesssim \nrm{\nb\bfc_t}_{L^2} + (\nrm{\bfc}_{L^\infty}+\nrm{\bfrho}_{L^\infty})(\nrm{\nb\bfc}_{L^2}+\nrm{\nb\bfrho}_{L^2}) + \nrm{\nb\bfu}_{L^\infty}\nrm{\nb\bfc}_{L^2} + \nrm{\bfu}_{L^\infty} \nrm{\lap\bfc}_{L^2}. 
	\end{split}
	\end{equation} Next, using $\nrm{\nb\bfc}_{L^2}^2\le\nrm{\bfc}_{L^2} \nrm{\bfc}_{H^2} $, we obtain \begin{equation}\label{eq:c-H2}
	\begin{split}
	\frac{1}{2}\frac{d}{dt}\nrm{\bfc_t}_{L^2}^2 +\frac{1}{2} \nrm{\nb\bfc_t}_{L^2}^2  &\lesssim \left( \nrm{\bfrho}_{L^\infty} + \nrm{\bfc}_{L^\infty}\nrm{\bfu}_{L^\infty} + \nrm{\nb\bfc}_{L^\infty}\nrm{\bfrho}_{L^\infty} + \nrm{\bfc}_{L^\infty}^2 \nrm{\bfrho}_{L^\infty}^2 \right)\\
	&\quad \times \left( \nrm{\bfc_t}_{L^2}^2+\nrm{\nb\bfrho}_{L^2}^2 + (1 + \nrm{\bfrho}_{L^\infty} + \nrm{\bfu}_{L^\infty}^2)\nrm{\bfc}_{L^2}^2 \right). 
	\end{split}
	\end{equation} On the other hand, in the $\nb\bfu$ equation we easily get from \eqref{eq:pressure} that { \begin{equation}\label{eq:omg-H2}
	\begin{split}
	\frac{d}{dt} \nrm{\nb\bfu}_{L^2}^2 & \lesssim \left(\nrm{\nb \bfu\cdot\nb\bfu }_{L^2} + \nrm{\nb^2\bfp}_{L^2} + \nrm{\nb(\bfrho\nb\phi)}_{L^2} \right) \nrm{\nb\bfu}_{L^2} \\
	&\lesssim  \nrm{\nb \bfu}_{L^\infty} \nrm{\bfu}_{H^{1}}^2 + \nrm{\bfrho}_{H^{1}}  \nrm{\nb\bfu}_{L^2}.
	\end{split}
	\end{equation}  }
	Collecting the estimates \eqref{eq:rho-H2}, \eqref{eq:c-H2}, and \eqref{eq:omg-H2} together with \eqref{eq:rho-H2-prime}, we obtain (upon choosing sufficiently small $\eps>0$) \begin{equation}\label{eq:KSF-bdd-H2}
	\begin{split}
	\frac{d}{dt} \left( \nrm{\nb\bfrho}_{L^2}^2+\nrm{\bfc_t}_{L^2}^2+{\nrm{\nb\bfu}_{L^{2}}^2} \right)\lesssim \left(1 + \nrm{\bfc}_{W^{2,\infty}} + \nrm{\bfrho}_{L^\infty} + \nrm{\bfu}_{W^{1,\infty}} \right)^4 \left( \nrm{\bfrho}_{H^1}^2+\nrm{\bfc_t}_{L^2}^2+ \nrm{\bfc}_{L^2} +{\nrm{\bfu}_{H^1}^2} \right) . 
	\end{split}
	\end{equation}
	
		\medskip
	
	\noindent \textit{(iii) $H^4$ estimate}: Taking the Laplacian to the equation for $\rd_t\nb\bfrho$, we obtain
	\begin{equation}\label{eq:KSF-bdd-der-lap1}
	\begin{split}
	\rd_t \nb\lap\bfrho + (\bfu+\nb\bfc)\cdot\nb (\nb\lap\bfrho)& = \lap( -\nb\bfrho \lap\bfc - \bfrho\cdot\nb\lap\bfc - (\nb\bfu+\nb^2\bfc)\cdot\nb\bfrho  ) \\
	&\quad  - \lap(\bfu+\nb\bfc) \cdot \nb(\nb\bfrho) - 2 \sum_i \rd_i(\bfu+\nb\bfc)\cdot\nb \nb \rd_i\bfrho .
	\end{split}
	\end{equation} We shall find the following elementary inequality useful: \begin{equation*}
	\begin{split}
	\nrm{\lap f}_{L^4}^2 \lesssim  {\nrm{\nb f}_{H^{2}} \nrm{f}_{W^{1,\infty}}}. 
	\end{split}
	\end{equation*} 
	This gives us  \begin{equation*}
	\begin{split}
	\left| \int  \lap(\bfu+ \nb\bfc) \cdot \nb \rd_i\bfrho \rd_i\lap\bfrho  \right| \lesssim ( \nrm{\bfu}_{W^{1,\infty}} + \nrm{\bfc}_{W^{2,\infty}} + \nrm{\bfrho}_{W^{1,\infty}} ) ( \nrm{\bfu}_{H^3}^2 + \nrm{\bfc}_{H^4}^2 + \nrm{\bfrho}_{H^3}^2 ).
	\end{split}
	\end{equation*} Then, it allows us to bound \begin{equation*}
	\begin{split}
	&\left|\int  ( - \lap (\nb\bfu+\nb^2\bfc)\cdot\nb\bfrho  -   \lap(\bfu+\nb\bfc) \cdot \nb(\nb\bfrho) - 2 \sum_i \rd_i(\bfu+\nb\bfc)\cdot\nb \nb \rd_i\bfrho  )  \nb\lap\bfrho \right| \\
	&\quad  \lesssim ( \nrm{\bfu}_{W^{1,\infty}} + \nrm{\bfc}_{W^{2,\infty}} + \nrm{\bfrho}_{W^{1,\infty}} ) ( \nrm{\bfu}_{H^3}^2 + \nrm{\bfc}_{H^4}^2 + \nrm{\bfrho}_{H^3}^2 ). 
	\end{split}
	\end{equation*} Proceeding similarly, we obtain \begin{equation*}
	\begin{split}
	\left| \int \lap(-\nb\bfrho \lap\bfc - \bfrho\cdot\nb\lap\bfc ) \cdot \nb\lap\bfrho \right| \lesssim 
	( \nrm{\bfc}_{W^{2,\infty}} + \nrm{\bfrho}_{W^{1,\infty}} ) ( \nrm{\bfc}_{H^4}^2 + \nrm{\bfrho}_{H^3}^2 ) + \nrm{\bfrho}_{L^\infty} \nrm{\nb\lap^2\bfc}_{L^2} \nrm{\nb\lap\bfrho}_{L^2} 
	\end{split}
	\end{equation*} This gives the estimate \begin{equation}\label{eq:rho-H4}
	\begin{split}
	\frac{1}{2} \frac{d}{dt} \nrm{\nb\lap\bfrho}_{L^2}^2 - \eps \nrm{\nb\lap^2\bfc}_{L^2}^2 &  \lesssim ( \nrm{\bfu}_{W^{1,\infty}} + \nrm{\bfc}_{W^{2,\infty}} + \nrm{\bfrho}_{W^{1,\infty}} ) ( \nrm{\bfu}_{H^3}^2 + \nrm{\bfc}_{H^4}^2 + \nrm{\bfrho}_{H^3}^2 ) \\
	&\quad  +  C_\eps \nrm{\bfrho}_{L^\infty}^2 \nrm{\nb\lap\bfrho}_{L^2}^2 .
	\end{split}
	\end{equation} Let us now estimate $\bfrho_{tt}$ and $\bfu_{tt}$. We have: \begin{equation*}
	\begin{split}
	\bfrho_{tt} = - \bfu_t \cdot\nb\bfrho - \bfu\cdot\nb\bfrho_t - \nb\bfrho_t\cdot\nb\bfc - \nb\bfrho\cdot\nb\bfc_t - \bfrho_t\lap\bfc - \bfrho\lap\bfc_t 
	\end{split}
	\end{equation*} recalling that \begin{equation*}
	\begin{split}
	\nrm{\bfu_t}_{L^2}\lesssim \nrm{\nb\bfu}_{L^\infty}\nrm{\bfu}_{L^2} + \nrm{\bfrho}_{L^2}
	\end{split}
	\end{equation*} and observing that \begin{equation*}
	\begin{split}
	\nrm{\bfrho_t}_{H^1} \lesssim ( \nrm{\bfu}_{W^{1,\infty}} + \nrm{\bfc}_{W^{2,\infty}} + \nrm{\bfrho}_{W^{1,\infty}} ) (\nrm{ \bfrho}_{H^2} + \nrm{ \bfc}_{H^3})  
	\end{split}
	\end{equation*} holds, we can bound \begin{equation*}
	\begin{split}
	&\nrm{- \bfu_t \cdot\nb\bfrho - \bfu\cdot\nb\bfrho_t - \nb\bfrho_t\cdot\nb\bfc - \nb\bfrho\cdot\nb\bfc_t - \bfrho_t\lap\bfc - \bfrho\lap\bfc_t }_{L^2} \\ 
	&\quad \lesssim \left(  \nrm{\bfc}_{W^{2,\infty}} + \nrm{\bfrho}_{W^{1,\infty}} + \nrm{\bfu}_{W^{1,\infty}} \right)^2(\nrm{\bfu}_{L^2} + \nrm{\bfrho}_{H^2} + \nrm{\bfc}_{H^3} + \nrm{\bfc_t}_{H^2}) .
	\end{split}
	\end{equation*} On the other hand, we have \begin{equation*}
	\begin{split}
	\nrm{\bfu_{tt}}_{L^2} \lesssim \nrm{\bfu}_{W^{1,\infty}}^2 \nrm{\bfu}_{H^2} + \nrm{\bfrho_t}_{L^2}.
	\end{split}
	\end{equation*} Moreover, using the equation for $\bfc_{tt}$ 
	\begin{equation}\label{eq:KSF-bdd-der-lap2}
	\begin{split}
	\rd_t \bfc_{tt} + \bfu\cdot\nb \bfc_{tt} = \lap \bfc_{tt} - \bfrho \bfc_{tt} -2 \bfrho_t\bfc_t   - \bfc\bfrho_{tt} - \bfu_{tt}\cdot\nb\bfc - \bfu_t\cdot\nb\bfc_t ,
	\end{split}
	\end{equation} we can estimate $\nrm{\nb\lap^2\bfc}_{L^2}$ in terms of $\nrm{\nb\bfc_{tt}}_{L^2}$ and $\nrm{\bfc_t}_{H^2}$, $\nrm{\bfc}_{H^4}$ in terms of $\nrm{\bfc_{tt}}_{L^2}$ (modulo lower order terms). We omit the details for this, but the proof is parallel to the estimate  $\nrm{\bfc}_{H^2}\lesssim \nrm{\bfc_t}_{L^2} + \cdots$ given in the above. Then, we have \begin{equation}\label{eq:c-H4}
	\begin{split}
	&\frac{1}{2} \frac{d}{dt} \nrm{\bfc_{tt}}_{L^2}^2 + \frac{1}{2} \nrm{\nb\bfc_{tt}}_{L^2}^2 \\
	 &\quad \lesssim \left( 1 +  \nrm{\bfc}_{W^{2,\infty}} + \nrm{\bfrho}_{W^{1,\infty}} + \nrm{\bfu}_{W^{1,\infty}} \right)^4(\nrm{\bfu}_{H^3} + \nrm{\bfrho}_{H^3} + \nrm{\bfc_{tt}}_{L^2} + \nrm{\bfc}_{L^2}) .
	\end{split}
	\end{equation}
	Next, taking the Laplacian to the both sides of the equation for $\nb\bfu$, 
		\begin{equation}\label{eq:KSF-bdd-der-lap3}
	\begin{split}
{\rd_t   \lap \nb  \bfu+ \bfu\cdot\nb \nb \lap \bfu= - \lap(\nb\bfu \cdot\nb\bfu) - \lap\nb^2 \bfp + \lap\nb(\bfrho\nb\phi) - \sum_{i} \rd_i  \bfu \rd_i \nb\nb\bfu - \lap \bfu \cdot\nb \nb\bfu  .}
	\end{split}
	\end{equation} Then, we may estimate \begin{equation}\label{eq:omg-H4}
	\begin{split}
	{\frac{d}{dt} \nrm{ \lap \nb  \bfu }_{L^2}^2 \lesssim \nrm{\bfrho}_{H^3} \nrm{\bfu}_{H^{3}} + \nrm{\bfu}_{W^{1,\infty}} \nrm{\bfu}_{H^3}^2. }
	\end{split}
	\end{equation}  
Therefore, from \eqref{eq:rho-H4}, \eqref{eq:c-H4}, and \eqref{eq:omg-H4}, we obtain that \begin{equation}\label{eq:KSF-bdd-H4}
	\begin{split}
	&\frac{d}{dt} \left( \nrm{\nb\lap\bfrho}_{L^2}^2 + \nrm{\bfc_{tt}}_{L^2}^2 + \nrm{\lap\nb \bfu}_{L^2}^2  \right)\\
	 &\quad \lesssim \left( 1 +  \nrm{\bfc}_{W^{2,\infty}} + \nrm{\bfrho}_{W^{1,\infty}} + \nrm{\bfu}_{W^{1,\infty}} \right)^4(\nrm{\bfu}_{H^3} + \nrm{\bfrho}_{H^3} + \nrm{\bfc_{tt}}_{L^2} + \nrm{\bfc}_{L^2}) .
	\end{split}
	\end{equation} Then, combining \eqref{eq:KSF-bdd-L2}, \eqref{eq:KSF-bdd-H2}, and \eqref{eq:KSF-bdd-H4}, \begin{equation*}
	\begin{split}
	&\frac{d}{dt} \left( \nrm{\bfrho}_{H^3}^2 +\nrm{\bfc}_{L^2}^2+\nrm{\bfc_{t}}_{L^2}^2+ \nrm{\bfc_{tt}}_{L^2}^2 + {\nrm{\bfu}_{H^3}^2}  \right)\\
	&\quad \lesssim \left(1+ \nrm{\bfrho}_{H^3}^2 +\nrm{\bfc}_{L^2}^2+\nrm{\bfc_{t}}_{L^2}^2 + \nrm{\bfc_{tt}}_{L^2}^2 + {\nrm{\bfu}_{H^3}^2}   \right)^6. 
	\end{split}
	\end{equation*} Hence, there exists $T>0$ depending only on $\nrm{\bfrho}_{H^3}^2 +\nrm{\bfc}_{L^2}^2+\nrm{\bfc_{t}}_{L^2}^2 + \nrm{\bfc_{tt}}_{L^2}^2 + {\nrm{\bfu}_{H^3}^2} $ at $t = 0$ such that on $[0,T]$, \begin{equation*}
	\begin{split}
	(\nrm{\bfrho}_{H^3}^2 +\nrm{\bfc}_{L^2}^2+\nrm{\bfc_{t}}_{L^2}^2 + \nrm{\bfc_{tt}}_{L^2}^2 + {\nrm{\bfu}_{H^3}^2})(t)  {\le 2} (\nrm{\bfrho}_{H^3}^2 +\nrm{\bfc}_{L^2}^2+\nrm{\bfc_{t}}_{L^2}^2 + \nrm{\bfc_{tt}}_{L^2}^2 + {\nrm{\bfu}_{H^3}^2})(0).  
	\end{split}
	\end{equation*} This completes the proof of a priori estimates. The blow-up criterion follows immediately from the alternative estimate \begin{equation*}
	\begin{split}
	&\frac{d}{dt} \left( \nrm{\bfrho}_{H^3}^2 +\nrm{\bfc}_{L^2}^2+\nrm{\bfc_{t}}_{L^2}^2+ \nrm{\bfc_{tt}}_{L^2}^2 + {\nrm{\bfu}_{H^3}^2}\right)\\
	&\quad \lesssim \left( 1 +  \nrm{\bfc}_{W^{2,\infty}} + \nrm{\bfrho}_{W^{1,\infty}} + \nrm{\bfu}_{W^{1,\infty}} \right)^4\left( \nrm{\bfrho}_{H^3}^2 +\nrm{\bfc}_{L^2}^2+\nrm{\bfc_{t}}_{L^2}^2 + \nrm{\bfc_{tt}}_{L^2}^2 + {\nrm{\bfu}_{H^3}^2} \right)^2. 
	\end{split}
	\end{equation*} 
	
	\medskip
	
	\noindent {We now deal with the existence of a solution with claimed regularity. Given $(\bfrho_0,\bfc_0,\bfu_0)$ satisfying the assumptions of the theorem, we first mollify it in a way that for each $\eps>0$, the triple $(\bfrho_0^\eps,\bfc_0^\eps,\bfu_0^\eps)$ is $C^\infty$--smooth in $\Omg$,  satisfies the compatibility conditions, and converge to $(\bfrho_0,\bfc_0,\bfu_0)$ strongly in the norm $H^3\times H^4\times H^3$ as $\eps\to 0$. Furthermore, we may assume that \begin{equation*}
		\begin{split}
			\left(  \nrm{\bfrho_0^\eps}_{H^3} + \nrm{\bfc_0^\eps}_{H^4} + \nrm{\bfu_0^\eps}_{H^{3}}  \right) \le 2 \left(  \nrm{\bfrho_0}_{H^3} + \nrm{\bfc_0}_{H^4} + \nrm{\bfu_0}_{H^{3}}  \right). 
		\end{split}
	\end{equation*} This mollification is done simply to ensure that all the objects that we are dealing with are $C^\infty$--smooth. Then, we now fix some small $\eps>0$ and build a sequence of approximations $(\bfrho^{(n,\eps)},\bfc^{(n,\eps)},\bfu^{(n,\eps)})_{n\ge0}$, which will be shown to be uniformly bounded in a time interval $[0,T]$. In the case $n = 0$, we simply set  $(\bfrho^{(0,\eps)},\bfu^{(0,\eps)}) = (\bfrho_0^\eps,\bfu_0^\eps)$ for all $t\ge 0$ and $\bfc^{(0,\eps)}$ be the solution of the heat equation \begin{equation*}
		\begin{split}
			\rd_t \bfc^{(0,\eps)} = \lap \bfc^{(0,\eps)}, \\
			\bfc^{(0,\eps)}(t=0) = \bfc_0^\eps,  
		\end{split}
	\end{equation*} with the Neumann boundary condition. Now, given $(\bfrho^{(n,\eps)},\bfc^{(n,\eps)},\bfu^{(n,\eps)})$ for some $n\ge 0$, we define $$(\bfrho^{(n+1,\eps)},\bfc^{(n+1,\eps)},\bfu^{(n+1,\eps)})$$ as the unique solution of the following \textit{linear} system of equations  \begin{equation}  \label{eq:KSF-approxseq}
	\left\{
	\begin{aligned} 
	&\rd_t\bfrho^{(n+1,\eps)} + (\bfu^{(n,\eps)} + \nb\bfc^{(n,\eps)})\cdot\nb\bfrho^{(n+1,\eps)}  = -\lap\bfc^{(n,\eps)} \bfrho^{(n+1,\eps)},\\
	&\rd_t\bfc^{(n+1,\eps)} + \bfu^{(n,\eps)}\cdot\nb\bfc^{(n+1,\eps)} = \lap\bfc^{(n+1,\eps)} - \bfrho^{(n,\eps)} \bfc^{(n+1,\eps)},\\
	&\rd_t \bfu^{(n+1,\eps)} + \bfu^{(n,\eps)}\cdot\nb\bfu^{(n+1,\eps)} + \nb \bfp^{(n+1,\eps)} = \bfrho^{(n,\eps)}\nb\phi, \\
	&\nb \cdot \bfu^{(n+1,\eps)} = 0, 
	\end{aligned}
	\right. 
	\end{equation} with initial data $(\bfrho_0^\eps,\bfc_0^\eps,\bfu_0^\eps)$, the Neumann boundary condition for $\bfc^{(n+1,\eps)}$, and the non-penetration boundary condition for $\bfu^{(n+1,\eps)}$. The existence of $\bfrho^{(n+1,\eps)}$ and $\bfu^{(n+1,\eps)}$  solving \eqref{eq:KSF-approxseq} can be proved  by integrating along the characteristics defined by $\bfu^{(n,\eps)} + \nb\bfc^{(n,\eps)}$ and $\bfu^{(n,\eps)}$, respectively. Lastly, the existence of a smooth solution $\bfc^{(n+1,\eps)}$ follows from the theory of parabolic equations (see for instance Ito \cite{Ito}). While these approximate solutions are defined globally in time, we need to obtain Sobolev estimates which are uniform in $n$. The arguments in this step follow closely the proof of a priori estimates we have obtained earlier. 

	To prove the uniform bound, we proceed by induction: assume that for all $0 \le k\le n$, there exists some $T>0$ such that \begin{equation*}
		\begin{split}
			M_{k} & := \sup_{0\le t \le T} \left(  \nrm{\bfrho^{(k,\eps)}}_{H^3} + \nrm{\bfc^{(k,\eps)}}_{H^4} + \nrm{\bfu^{(k,\eps)}}_{H^{3}}  \right)(t) +  \nrm{\nb \bfc^{(k,\eps)} (t) }_{L^2([0,T);H^{4})} \\
			&\qquad \le 4 \left(  \nrm{\bfrho_0}_{H^3} + \nrm{\bfc_0}_{H^4} + \nrm{\bfu_0}_{H^{3}}  \right).
		\end{split}
	\end{equation*} It is clear that this estimate holds for the base case $k=0$, with any $T>0$. Next, under this induction hypothesis, we can prove the same bound for $k=n+1$ by proceeding similarly as in the proof of a priori estimates, by possibly taking a smaller $T>0$ but in a way independent of $n$. To demonstrate this in the case of $\bfrho^{(n+1,\eps)}$, we have \begin{equation*}
	\begin{split}
		\frac{d}{dt} \nrm{\bfrho^{(n+1,\eps)} }_{H^{3}}^{2} \lesssim (  \nrm{\bfc^{(n,\eps)}}_{H^4} + \nrm{\bfu^{(n,\eps)}}_{H^{3}}   ) \nrm{\bfrho^{(n+1,\eps)} }_{H^{3}}^{2} + \nrm{ \lap\bfc^{(n,\eps)} }_{H^{3}} \nrm{\bfrho^{(n+1,\eps)} }_{H^{3}}^{2}.
	\end{split}
\end{equation*} Integrating in time, \begin{equation*}
\begin{split}
	 \sup_{0\le t \le T} \nrm{\bfrho^{(n+1,\eps)} }_{H^{3}} & \le  \nrm{\bfrho^{(\eps)}_{0} }_{H^{3}}\exp\left( \int_0^T C( \nrm{\bfc^{(n,\eps)}}_{H^4} + \nrm{\bfu^{(n,\eps)}}_{H^{3}}  +  \nrm{ \lap\bfc^{(n,\eps)} }_{H^{3}} ) dt     \right) \\
	 & \le \nrm{\bfrho^{(\eps)}_{0} }_{H^{3}}\exp\left( CTM_{n} + CT^{\frac12} M_{n}  \right) < 2\nrm{\bfrho^{(\eps)}_{0} }_{H^{3}}
\end{split}
\end{equation*} by taking $T>0$ smaller if necessary, in a way depending only on the quantity $\nrm{\bfrho_0}_{H^3} + \nrm{\bfc_0}_{H^4} + \nrm{\bfu_0}_{H^{3}}$. In this way, we obtain that the sequence $(\bfrho^{(n,\eps)},\bfc^{(n,\eps)},\bfu^{(n,\eps)})_{n\ge0}$ is uniformly bounded in $L^\infty([0,T]; H^3\times H^4\times H^3)$. Furthermore, one can see that the sequence $(\rd_t\bfrho^{(n,\eps)},\rd_t\bfc^{(n,\eps)},\rd_t\bfu^{(n,\eps)})_{n\ge0}$ is uniformly bounded in $L^\infty([0,T]; L^2\times L^2\times L^2)$. Applying the Aubin--Lions lemma and by passing to a subsequence, we can extract a convergent subsequence strongly in $L^\infty([0,T]; H^2\times H^3\times H^2)$. Let us denote the limit by $(\bfrho^\eps, \bfc^\eps, \bfu^\eps)$, which again belongs to the space $L^\infty([0,T]; H^3\times H^4\times H^3)$. The strong convergence in $H^2\times H^3\times H^2$ of  $(\bfrho^{(n,\eps)},\bfc^{(n,\eps)},\bfu^{(n,\eps)})_{n\ge0}$ to  $(\bfrho^\eps, \bfc^\eps, \bfu^\eps)$ is enough to guarantee pointwise convergence of each term in \eqref{eq:KSF-bdd} as $\eps\to 0$. Therefore, it follows that $(\bfrho^\eps, \bfc^\eps, \bfu^\eps)$ is a solution to \eqref{eq:KSF-bdd} with uniform bound \begin{equation*}
\begin{split}
	\sup_{0\le t \le T} \left(  \nrm{\bfrho^{\eps}}_{H^3} + \nrm{\bfc^{\eps}}_{H^4} + \nrm{\bfu^{\eps}}_{H^{3}}  \right)(t) +  \nrm{\nb \bfc^{\eps} (t) }_{L^2([0,T);H^{4})} \le 4 \left(  \nrm{\bfrho_0}_{H^3} + \nrm{\bfc_0}_{H^4} + \nrm{\bfu_0}_{H^{3}}  \right).
\end{split}
\end{equation*} Then, we can similarly take a convergent subsequence $(\bfrho^{\eps_{k}}, \bfc^{\eps_{k}}, \bfu^{\eps_{k}})$ and denote the corresponding strong limit in $H^2\times H^3\times H^2$ by $(\bfrho,\bfc,\bfu)$. It is not difficult to check that this triple solves \eqref{eq:KSF-bdd} with prescribed initial data and satisfies the claimed regularity in the statement of the theorem. }

	\medskip

	\noindent {Finally, uniqueness can be proved easily: \textit{assuming} that there are two solutions $(\bfrho,\bfc,\bfu)$ and $(\widetilde{\bfrho},\widetilde{\bfc}, \widetilde{\bfu})$ belonging to $H^3\times H^4\times H^3$ on $[0,T]$ with the same initial data, one can close an $L^2$ estimate for the difference as in the uniqueness proof of Theorem \ref{thm:lwp-away} below. The  $H^3\times H^4\times H^3$--regularity of the solution is sufficient to justify the $L^2$ difference estimate.}
\end{proof}

\section{Local well-posedness in the fully inviscid case}\label{sec:inviscid}

In this section, we consider the Cauchy problem for \eqref{eq:KS} and \eqref{eq:KSF} in the fully inviscid case. We establish sufficient conditions on the initial data which guarantees local well-posedness. Let us explain how the rest of this section is organized. To begin with, in Subsection \ref{subsec:lwp-nonvanishing}, we state and prove Theorem \ref{MainThm2}, which is local well-posedness in the fully inviscid \eqref{eq:KSF} with non-vanishing data. Within this Subsection, we perform a simple linear analysis which motivates our choice of modified good variables. Moreover, some discussion related with removing the non-vanishing assumption is given. Two versions of Theorem \ref{MainThm3}, which are local well-posedness results for fully inviscid \eqref{eq:KS} with vanishing initial data, are stated and proved respectively in Subsections \ref{subsec:v1} and \ref{subsec:v2}.

\subsection{Well-posedness with non-vanishing $\bfrho$ and $\bfc$}\label{subsec:lwp-nonvanishing}

We consider the system \eqref{eq:KSF} in $d$-dimensional domains of the form $\Omg = \bbT^k \times \bbR^{d-k}$ for any $0 \le k \le d$. Our first main result states that, if initially $\bfrho_0$ and $\bfc_0$ are bounded away from 0, the fully inviscid system is locally well-posed in sufficiently high Sobolev spaces. 

\begin{theorem}[Well-posedness away from 0]\label{thm:lwp-away}
	Assume that $D_u = D_\rho = D_c = 0$, \begin{equation*}
	\begin{split}
	az^\gmm \le \frac{k(z)}{\chi(z)}, \quad 0 \le z \le 1
	\end{split}
	\end{equation*} for some $a>0, \gmm\ge 0$, and $$\phi \in H^{m+1}(\Omg), \quad \chi, k \in H^{m+1}(\bbR_+),$$ for some $2 + \frac{d}{2} < m \le +\infty$. 
	Assume further that the initial data $(\bfrho_0, \bfc_0,\bfu_0)$ satisfies $\bfrho_0, \bfc_0 \in L^\infty(\Omg)$ and \begin{equation*} 
	\begin{split}
	& \bfrho_0(x) \ge \underline{r}_0 > 0 , \quad \bfc_0(x) \ge \underline{c}_0 > 0, \quad \mathrm{div}\, \bfu_0 = 0 
	\end{split}
	\end{equation*} for some constants $ \underline{r}_0, \underline{c}_0 $ and \begin{equation*}
	\begin{split}
	\nabla\bfrho_0 \in H^{m-1}(\Omg), \quad \bfu_0, \nabla\bfc_0 \in H^m(\Omg)
	\end{split}
	\end{equation*} Then the system \eqref{eq:KSF} is locally-well posed; more precisely, there exists some $T > 0$ depending only on the initial data such that there is a unique solution $(\bfrho,\bfc,\bfu)$ to \eqref{eq:KSF} satisfying the initial condition and \begin{equation*}
	\begin{split}
	\bfrho, \bfc \in L^\infty([0,T); L^\infty(\Omg)), \quad  \nabla\bfrho \in C([0,T); H^{m-1}(\Omg)), \quad \bfu, \nabla\bfc \in C([0,T); H^m(\Omg)). 
	\end{split}
	\end{equation*}
\end{theorem}

\begin{remark}
	Taking $\bfu_0 = 0$ and $\phi = 0$, the above result gives in particular that the system \eqref{eq:KS} is locally-well posed in the fully inviscid case $(D_\rho = D_c = 0)$ for initial data $(\bfrho_0, \bfc_0)$ which is away from 0. Moreover, it is not difficult to show that the a priori estimates we obtain for the fully inviscid case carries over to partially viscous cases, namely when some of the viscosity constants $D_u, D_\rho, D_c$ are positive. 
\end{remark}

\begin{remark}
	A slightly technical issue appears in the case of an unbounded domain, since then the functions $\bfrho,\bfc$ being bounded away from 0 forces in particular that they do not belong to any $L^p$ for finite $p$. It is then natural to require that there exist time-dependent constants $\bfrho_\infty(t)$ and $\bfc_\infty(t)$ such that \begin{equation*}
	\begin{split}
	\bfrho(t,\cdot) - \bfrho_{\infty}(t) \in H^{m}, \quad \bfc(t,\cdot) - \bfc_{\infty}(t) \in H^{m+1}. 
	\end{split}
	\end{equation*} Then one can easily guarantee (formally) that the solution must be uniformly bounded in space if the initial data is. Indeed, it is easy to see (by evaluating at infinity) directly from \eqref{eq:KSF} that \begin{equation*}
	\begin{split}
	\frac{d}{dt} \bfrho_\infty = 0, \quad \frac{d}{dt} \bfc_\infty = - k(\bfc_\infty) \bfrho_\infty
	\end{split}
	\end{equation*} which determines $(\bfrho_\infty(t),\bfc_\infty(t))$ in terms of $(\bfrho_\infty(0), \bfc_\infty(0))$. Hence one may simply work with decaying functions $\bfrho(t,\cdot) - \bfrho_{\infty}(t)$ and $ \bfc(t,\cdot) - \bfc_{\infty}(t)$. On the other hand, we can simply require $\bfu(t,\cdot) \in {L}^\infty$, since the contribution from $\bfrho_\infty$ into $\rd_t\bfu$ can be absorbed into the pressure term. A different way to handle this non-decaying issue is to altogether avoid putting $\bfrho, \bfc$ in $L^2$ and simply use $L^\infty$ instead. In the following, we shall neglect this issue and just work in $\bbT^d$.
\end{remark}

\subsubsection{Discussion}

The non-vanishing assumption on $\bfrho$ and $\bfc$ stems naturally from our method of proof, which utilizes the ``modified variables'' $\sqrt{\bfc}\rd^m\bfrho$ and $\sqrt{\bfrho} \rd^{m+1}\bfc$. Although we expect that such an assumption cannot be entirely omitted, we present two different results in which non-vanishing assumption is relaxed. The first one is simply employing the modified variables $(\sqrt{\bfrho})^{-1} \rd^m\bfrho$ and $(\sqrt{\bfc})^{-1} \rd^{m+1}\bfc$. A disadvantage in this approach is that, to obtain very regular solutions (i.e. well-posedness for $m$ large), one needs to assume that whenever $\bfc_0$ or $\bfrho_0$ vanishes, it must do so with a high order (proportional to $m$). The other approach is to \textit{specify} the profiles of $\bfrho_0$ and $\bfc_0$ near the points of vanishing; e.g. $\bfrho_0(x) \simeq A|x-x_0|^2$. There is nothing special about the data being quadratic at its zeroes, and the same method can be applied to smooth data which vanishes with higher order. While the assumption on the initial data is more rigid, propagating $C^\infty$-smoothness is not difficult in this case. Moreover, in this latter setting, we can prove finite-time singularity formation for \eqref{eq:KS}; the $C^2$-norm of $\bfc(t,\cdot)$ becomes infinite in finite time. 

\subsubsection{Linear analysis}

In the fully inviscid case $D_u = D_\rho = D_c = 0$, there are serious difficulties in closing $H^m$ a priori estimates. To see whether there is a chance of the inviscid system to be well-posed, we first consider \eqref{eq:KS} in the 1D case: 
\begin{equation}\label{eq:KS-1D}
\left\{
\begin{aligned}
\rd_t \bfrho & =  -\rd_x ( \bfrho \rd_x \bfc), \\
\rd_t \bfc & =  - \bfc\bfrho  .
\end{aligned}
\right.
\end{equation} Note that we have taken $k(z) = z, \chi(z) = 1$ for simplicity. The term $-\rd_x(\bfrho\rd_x\bfc)$ on the right hand side for the $\bfrho$-equation seems like it incurs loss of derivatives. To see more clearly the effect of this term, we take the linearization approach: while there are no non-trivial steady states (solutions independent of time), the next simplest solutions are given by $(\bar{\bfrho}, \bar{\bfc} e^{-\bar{\bfrho}t})$ where $\bar{\bfrho},\bar{\bfc}$ are some positive constants. Fixing $\bar{\bfrho} = \bar{\bfc} = 1$, writing \begin{equation*}
\begin{split}
\bfrho =1+ \tilde{\bfrho}, \quad \bfc = e^{-t} + \tilde{\bfc} ,
\end{split}
\end{equation*} and dropping quadratic terms in the perturbation, we arrive at  \begin{equation}\label{eq:KS1-lin}
\left\{
\begin{aligned}
\rd_t \tilde{\bfrho} & = -\rd_{xx}\tilde{\bfc} ,\\
\rd_t \tilde{\bfc} & = - e^{-t} \tilde{\bfrho} - \tilde{\bfc}.  
\end{aligned}
\right.
\end{equation} This linear system is well-posed: we have that the ``energy'' \begin{equation*}
\begin{split}
\nrm{\tilde{\bfrho}}_{L^2}^2 + e^t\nrm{\rd_x\tilde{\bfc}}_{L^2}^2
\end{split}
\end{equation*} is under control: to see this, we compute that \begin{equation*}
\begin{split}
\frac{1}{2}\frac{d}{dt} \left( \nrm{\tilde{\bfrho}}_{L^2}^2 + e^t \nrm{\rd_x\tilde{\bfc}}_{L^2}^2 \right) = - \frac{1}{2}e^t \nrm{\rd_x\tilde{\bfc}}_{L^2}^2 \le 0 . 
\end{split}
\end{equation*} Proceeding similarly for higher derivatives, we see that $\nrm{\rd_x^n\tilde{\bfrho}}_{L^2}^2 + e^t\nrm{\rd_x^{n+1,\eps}\tilde{\bfc}}_{L^2}^2$ decreases in time for any $n \ge 0$. This suggests that a suitable weighted norm of the solution could be under control  for the nonlinear evolution. It is also expected that the weight should be solution-dependent,  which naturally gives rise to the non-vanishing assumptions. 

Moving on to the case of \eqref{eq:KSF}, we consider the following 1D model system: \begin{equation}  \label{eq:KSF-inviscid-1D}
\left\{
\begin{aligned}
 &\rd_t\bfrho + \bfu\rd_x\bfrho = -\rd_x(\bfrho\rd_x\bfc) ,\\
 &\rd_t\bfc + \bfu\rd_x\bfc = -\bfc\bfrho,\\
 &\rd_t\bfu + \bfu\rd_x\bfu = \bfrho .
\end{aligned}
\right.
\end{equation} From the above, we know that it is natural to rewrite the equation in terms of $\bff = \rd_x\bfc$. Taking a derivative in the $\bfc$-equation and writing $D_t = \rd_t + \bfu\rd_x$ for simplicity, we obtain: \begin{equation*}
\left\{ 
\begin{aligned}
&D_t \bfrho = -\bfrho \rd_x \bff - \bff \rd_x\bfrho ,\\
&D_t\bff = - \bff \rd_x\bfu - \bff \bfrho - \bfc \rd_x\bfrho ,\\
&D_t\bfu = \bfrho. 
\end{aligned}
\right.
\end{equation*} A formal linearization around the state $\bfrho=\bff=\bfu = 1$ gives, after removing terms which do not lose derivatives,  \begin{equation}  \label{eq:KSF-inviscid-1D-linear}
\left\{
\begin{aligned}
{(\rd_t+\rd_x)}\bfrho & = - \rd_{x}\bff - \rd_x\bfrho  , \\
{(\rd_t+\rd_x)} \bff & = - \rd_x\bfu - \rd_x\bfrho , \\
{(\rd_t+\rd_x)} \bfu & = 0 . 
\end{aligned}
\right.
\end{equation} Then it is tempting to rewrite the above as \begin{equation}  \label{eq:illposed-1D-good}
\left\{
\begin{aligned}
{(\rd_t+\rd_x)}(\bfrho + \bfu ) & = - \rd_{x}(\bff - \bfu ) - \rd_x(\bfrho + \bfu ) , \\
{(\rd_t+\rd_x)} (\bff  - \bfu )& = - \rd_x(\bfrho + \bfu ), \\
{(\rd_t+\rd_x)} \bfu & = 0 . 
\end{aligned}
\right.
\end{equation} In this form, it is clear that \eqref{eq:KSF-inviscid-1D-linear} is well-posed in Sobolev spaces. This suggests that even the non-linear system could be well-posed for some delicate reason. Of course, in higher dimensions, one needs to take into account the effect of the pressure as well. With this in mind, we shall now give a proof of Theorem \ref{thm:lwp-away}. 

\begin{proof}[{\bf Proof of Theorem \ref{thm:lwp-away}}]
	
	For now, we take $\chi(z) = 1$, $k(z) = z$ and $\Omg = \bbT^d$ for simplicity and rewrite the inviscid \eqref{eq:KSF} system with $\bff = \nabla \bfc$: \begin{equation} \label{eq:KSF-inviscid}
	\left\{
	\begin{aligned}
	&\rd_t \boldsymbol{\rho}  + \bfu\cdot\nabla \bfrho = -\nabla \cdot (\bfrho\bff) , \\
	&\rd_t \bff + \bfu\cdot\nabla \bff = -[ \nb\bfu  ]^T \bff - \bfrho \bff - \bfc\nabla\bfrho , \\
	&\rd_t \bfu + \bfu\cdot\nabla \bfu + \nabla \bfp =  \bfrho\nabla \phi , \\
	& \mathrm{div}\, \bfu = 0. 
	\end{aligned}
	\right.
	\end{equation} Let us divide the proof into several steps. 
	
	\medskip
	
	\noindent \textit{(i) Regularity of the pressure}: Taking the divergence of the equation for $\bfu$, we obtain that \begin{equation*}
	\begin{split}
	\lap \bfp = -\sum_{1 \le i,j \le d} \rd_i \bfu^j \rd_j \bfu^i + \nabla\cdot(\bfrho \nb \phi). 
	\end{split}
	\end{equation*} We then have that \begin{equation*}
	\begin{split}
	\nabla \bfp = \nabla(-\lap)^{-1} \left(\sum_{1 \le i,j \le d} \rd_i \bfu^j \rd_j \bfu^i\right) - \nabla(-\lap)^{-1} \nb\cdot(\bfrho\nb\phi) . 
	\end{split}
	\end{equation*} We claim that for $m \ge \frac{d}{2} + 1$, \begin{equation*}
	\begin{split}
	\nrm{\nb \bfp}_{H^m} \lesssim \nrm{\bfu}_{H^m}^2 + \nrm{\nb\phi}_{H^m} \nrm{\bfrho}_{H^m}. 
	\end{split}
	\end{equation*} It is straightforward to bound the second term: since $\nb(-\lap)^{-1}\nb\cdot$ is an operator with bounded Fourier multiplier, \begin{equation*}
	\begin{split}
	\nrm{\nabla(-\lap)^{-1} \nb\cdot(\bfrho\nb\phi) }_{H^m} \lesssim \nrm{\bfrho\nb\phi}_{H^m} \lesssim \nrm{\nb\phi}_{H^m} \nrm{\bfrho}_{H^m}.
	\end{split}
	\end{equation*} On the other hand, for any $n \ge 1$, \begin{equation*}
	\begin{split}
	\nrm{ \nb (-\lap)^{-1} (\rd_i\bfu^j \rd_j\bfu^i) }_{\dot{H}^n}  \lesssim \nrm{\rd_i\bfu^j \rd_j \bfu^i}_{\dot{H}^{n-1}} \lesssim \nrm{\nabla\bfu}_{L^\infty} \nrm{\bfu}_{H^n}.
	\end{split}
	\end{equation*} Lastly, using incompressibility \begin{equation*}
	\begin{split}
	\nrm{ \nb (-\lap)^{-1}\sum_i (\rd_i\bfu^j \rd_j\bfu^i) }_{L^2} = \nrm{ \nb (-\lap)^{-1}\rd_i \sum_i (\bfu^j \rd_j\bfu^i) }_{L^2} \lesssim \nrm{\nb\bfu}_{L^\infty} \nrm{\bfu}_{L^2}. 
	\end{split}
	\end{equation*} The claim follows.

	\medskip
	
	\noindent \textit{(ii) Estimate of the infimum}: Assuming for a moment that a sufficiently regular solution exists, we now obtain a simple estimate on the variation of the infimum. For this, let $x^*(t)$ be any time-dependent continuous curve of infimum point of $\bfrho(t)$. 
	Dividing both sides of the $\bfrho$-equation by $\bfrho^2$, \begin{equation*}
	\begin{split}
	\rd_t (\frac{1}{\bfrho})  + (\bfu + \bff) \cdot\nabla (\frac{1}{\bfrho})  =  (\nb \cdot \bff) (\frac{1}{\bfrho}) . 
	\end{split}
	\end{equation*} Evaluating along the characteristics defined by $\bfu+\bff$, we obtain the estimate \begin{equation}\label{eq:rho-inf}
	\begin{split}
	\left| \frac{d}{dt} \nrm{\bfrho}_{inf}^{-1}  \right| \le \nrm{\lap\bfc}_{L^\infty}  \nrm{\bfrho}_{inf}^{-1}   . 
	\end{split}
	\end{equation} Similarly, from the equation for $\bfc$, one can show that \begin{equation}\label{eq:c-inf}
	\begin{split}
	\left| \frac{d}{dt}\nrm{\bfc}_{inf}^{-1}   \right| \le \nrm{\bfrho}_{L^\infty}\nrm{\bfc}_{inf}^{-1}   . 
	\end{split}
	\end{equation} 
	
	\medskip
	
	\noindent \textit{(iii) Good variables}: Still proceeding under the assumption of the existence of a sufficiently smooth solution, we introduce the good variables and perform a priori estimates. For each $m\ge 0$, we fix an $m$-th order partial derivative, and define the ``good'' variables as follows: \begin{equation*}
	\begin{split}
	R_m &:= \rd^m \bfrho + \frac{1}{\bfc} \bff \cdot \rd^m \bfu ,\\
	F_m &:= \rd^m \bff - \frac{1}{\bfrho\bfc}  \bff (\bff \cdot \rd^m \bfu). 
	\end{split}
	\end{equation*} We introduce notation $(\calO_m)^k$ to denote expressions that can be bounded in $L^2$ by a constant multiple of $(1+\nrm{\bfu}_{H^m} + \nrm{\bfrho}_{H^m} + \nrm{\bfc}_{H^{m+1}} + \nrm{\bfrho}_{inf}^{-1} + \nrm{\bfc}_{inf}^{-1} )^k$. From the equation for $\bfu$ (using the previous bound for $\nb\bfp$), we have \begin{equation*}
	\begin{split}
	\rd_t \rd^m\bfu + \bfu \cdot\nb \rd^m \bfu = (\calO_m)^2.  
	\end{split}
	\end{equation*} Similarly, we may write \begin{equation*}
	\begin{split}
	\rd_t \rd^m\bfrho + \bfu \cdot \nb \rd^m\bfrho = - \bfrho \nb \cdot \rd^m\bff - \bff \cdot \nabla \rd^m \bfrho + (\calO_m)^2 
	\end{split}
	\end{equation*} and \begin{equation*}
	\begin{split}
	\rd_t \rd^m \bff + \bfu \cdot \nabla \rd^m\bff = - \nb (\bff \cdot \rd^m\bfu) - \bfc \nb \rd^m\bfrho + (\calO_m)^2. 
	\end{split}
	\end{equation*} We first massage the equation for $\rd_t\rd^m\bfrho$: \begin{equation*}
	\begin{split}
	\rd_t( \rd^m\bfrho + \frac{1}{\bfc} \bff \cdot \rd^m\bfu) +  (\bfu\cdot\nb)( \rd^m\bfrho + \frac{1}{\bfc} \bff \cdot \rd^m\bfu) &= \rd_t(\frac{1}{\bfc}\bff) \cdot \rd^m\bfu - \bfrho\nb\cdot\rd^m\bff + \frac{1}{\bfc}\bff \cdot( (\bff\cdot\nb) \rd^m\bfu) \\
	&\quad - (\bff\cdot\nb) ( \rd^m\bfrho + \frac{1}{\bfc} \bff \cdot \rd^m\bfu) +  {(\calO_m)^6}. 
	\end{split}
	\end{equation*} 
	Note that we can rewrite the second and third terms on the right hand side as (repeated indices being summed) \begin{equation*}
	\begin{split}
	- \bfrho\nb\cdot\rd^m\bff + \frac{1}{\bfc}\bff \cdot( (\bff\cdot\nb) \rd^m\bfu)  & = -\bfrho \left( \rd^m\rd_j\bff^j  - \frac{1}{\bfc\bfrho} \bff^\ell \bff^j \rd_j \rd^m \bfu^\ell   \right) \\
	& = -\bfrho \rd_j \left( \rd^m\bff^j  - \frac{1}{\bfc\bfrho} \bff^j  \bff\cdot  \rd^m \bfu   \right)  - \bfrho \rd_j(\frac{1}{\bfc\bfrho}) \bff^j \bff\cdot\rd^m\bfu + (\calO_m)^4.
	\end{split}
	\end{equation*} That is, we have \begin{equation}\label{eq:R_m}
	\begin{split}
	D_t R_m = -(\bff\cdot\nb) R_m - \bfrho \nabla \cdot F_m +   (\calO_m)^6.
	\end{split}
	\end{equation}  On the other hand, we have \begin{equation}\label{eq:F_m}
	\begin{split}
	D_t F_m = -\bfc \nb R_m +  (\calO_m)^6.
	\end{split}
	\end{equation} Multiplying both sides of \eqref{eq:R_m} by $\bfc R_m$ and integrating, we obtain that \begin{equation*}
	\begin{split}
	\left| \int \bfc\bfrho R_m\nabla \cdot F_m
	+ {\frac12} \frac{d}{dt} \nrm{\sqrt{\bfc}R_m}_{L^2}^{2} \right| &\lesssim \nrm{\nb(\bfc\bff)}_{L^\infty} \nrm{R_m}_{L^2}^2 +  \nrm{\bfc}_{L^\infty}\nrm{R_m}_{L^2}(\calO_m)^6 \\
	&\lesssim \nrm{R_m}_{L^2}^2 (\calO_m)^2 + \nrm{R_m}_{L^2} (\calO_m)^7 \\
	&\lesssim \nrm{R_m}_{L^2}^2 (\calO_m)^2 + (\calO_m)^{12}. 
	\end{split}
	\end{equation*} {Note that now we are abusing notation to write $(\calO_m)^k$ for quantities bounded by a constant multiple of $(1+\nrm{\bfu}_{H^m} + \nrm{\bfrho}_{H^m} + \nrm{\bfc}_{H^{m+1}} + \nrm{\bfrho}_{inf}^{-1} + \nrm{\bfc}_{inf}^{-1} )^k$.} Similarly, multiplying both sides of \eqref{eq:F_m} by $\bfrho F_m$ and integrating gives \begin{equation*}
	\begin{split}
	\left| \int \bfrho\bfc F_m \nb R_m + {\frac12}  \frac{d}{dt} \nrm{\sqrt{\bfrho} F_m}_{L^2}^2  \right| \lesssim \nrm{F_m}_{L^2}^2 (\calO_m)^2 +  (\calO_m)^{12}. 
	\end{split}
	\end{equation*} Combining the above, we have \begin{equation*}
	\begin{split}
	\frac{d}{dt}\left( \nrm{\sqrt{\bfc}R_m}_{L^2}^{2} +  \nrm{\sqrt{\bfrho} F_m}_{L^2}^2 \right) \lesssim  \left( \nrm{\sqrt{\bfc}R_m}_{L^2}^{2} +  \nrm{\sqrt{\bfrho} F_m}_{L^2}^2 \right) (\calO_m)^3 + (\calO_m)^{12}.  
	\end{split}
	\end{equation*} Now, from the equation for $\rd_t \rd^m\bfu$, we obtain \begin{equation*}
	\begin{split}
	\frac{d}{dt} \nrm{\rd^m\bfu}_{L^2}^2 \lesssim (\calO_m)^3. 
	\end{split}
	\end{equation*} For a stronger reason, we can derive \begin{equation*}
	\begin{split}
	\frac{d}{dt}\sum_{k=0}^m\left( \nrm{\sqrt{\bfc}R_k}_{L^2}^{2} +  \nrm{\sqrt{\bfrho} F_k}_{L^2}^2  + \nrm{\rd^k\bfu}_{L^2}^2  \right) \lesssim \left( \nrm{\sqrt{\bfc}R_m}_{L^2}^{2} +  \nrm{\sqrt{\bfrho} F_m}_{L^2}^2 \right) (\calO_m)^3 + (\calO_m)^{12}.  
	\end{split}
	\end{equation*}
	Next, it is not difficult to see that by defining \begin{equation*}
	\begin{split}
	Z_m := \sum_{k=0}^m\left( \nrm{\sqrt{\bfc}R_k}_{L^2}^{2} +  \nrm{\sqrt{\bfrho} F_k}_{L^2}^2  + \nrm{\rd^k\bfu}_{L^2}^2  \right) +  \nrm{\bfrho}_{inf}^{-1} + \nrm{\bfc}_{inf}^{-1}  + \nrm{\bff}_{L^\infty} , 
	\end{split}
	\end{equation*} we have \begin{equation*}
	\begin{split}
	\frac{d}{dt} Z_m \lesssim  (Z_m)^2 (1 + (Z_m)^2) (\calO_m)^8 + (\calO_m)^{12}. 
	\end{split}
	\end{equation*} We are in a position to close the estimates in terms of $Z_m$: recalling the definitions of $R_m$ and $F_m$, we have  \begin{equation*}
	\begin{split}
	\nrm{\bfrho}_{H^m} + \nrm{\bff}_{H^m} \lesssim \sum_{k=0}^m( \nrm{R_k}_{L^2} + \nrm{F_k}_{L^2}  ) + (1+ \nrm{\bfrho}_{inf}^{-1} + \nrm{\bfc}_{inf}^{-1})^2 \nrm{\bff}_{L^\infty}^2 \nrm{\bfu}_{H^m} \lesssim (1 + Z_m)^5 
	\end{split}
	\end{equation*} which gives $\calO_m \lesssim (1 + Z_m)^5$. Therefore, \begin{equation}\label{eq:apriori-inviscid}
	\begin{split}
	\frac{d}{dt} Z_m \lesssim ( 1 + Z_m)^{60}. 
	\end{split}
	\end{equation} Therefore, for $Z_m(0)<\infty$, there exists $T>0$ such that $Z_m(t) \le 2Z_m(0)$ for $t<T$. 
	
	\medskip
	
	\noindent \textit{(iv) Existence}: Existence of a solution can be shown using viscous approximations; for fixed $\epsilon>0$, we consider the viscous system \begin{equation} \label{eq:KSF-viscous}
	\left\{
	\begin{aligned}
	&\rd_t \bfrho^{(\epsilon)}  + \bfu^{(\epsilon)}\cdot\nabla \bfrho^{(\epsilon)} = -\nabla \cdot (\bfrho^{(\epsilon)}\bff^{(\epsilon)}) + \eps \lap \bfrho^{(\epsilon)} , \\
	&\rd_t \bff^{(\epsilon)} + \bfu^{(\epsilon)}\cdot\nabla \bff^{(\epsilon)} = -[ \nb\bfu^{(\epsilon)}  ]^T \bff^{(\epsilon)} - \bfrho^{(\epsilon)} \bff^{(\epsilon)} - \bfc\nabla\bfrho^{(\epsilon)} + \eps \lap \bff^{(\epsilon)} , \\
	&\rd_t \bfu^{(\epsilon)} + \bfu^{(\epsilon)}\cdot\nabla \bfu^{(\epsilon)} + \nabla \bfp^{(\epsilon)} =  \bfrho^{(\epsilon)}\nabla \phi + \eps\lap\bfu^{(\epsilon)} , \\
	& \mathrm{div}\, \bfu^{(\epsilon)} = 0 
	\end{aligned}
	\right.
	\end{equation}  with the same initial data $(\bfrho_0,\bff_0,\bfu_0)$. Existence of a smooth local in time solution to \eqref{eq:KSF-viscous} was established already in \cite{CKL1,CKL2}. Moreover, it is not difficult to show by direct computation that the solution satisfies the a priori estimate \eqref{eq:apriori-inviscid} uniformly in $\epsilon>0$. {Since this step is not entirely obvious, we repeat the infimum estimate for the viscous solutions in the case of $\bfrho^{(\epsilon)}$. (The case of $\bfc^{(\epsilon)}$ is similar.) Since $\bfrho^{(\epsilon)}$ is continuous in time, at least for a very small time interval, we have a uniform lower bound $\bfrho^{(\epsilon)}(t,x) \ge c_0>0$ for some constant $c_{0}$ from the positivity of the initial data. Therefore, with the identity \begin{equation*}
		\begin{split}
			\lap \frac{1}{\bfrho^{(\epsilon)}} = - \frac{\lap \bfrho^{(\epsilon)}}{(\bfrho^{(\epsilon)})^2} + 2 \frac{\nb \bfrho^{(\epsilon)}}{\bfrho^{(\epsilon)}} \cdot \nb \frac{1}{\bfrho^{(\epsilon)}}, 
		\end{split}
	\end{equation*}we can compute on this time interval \begin{equation*}
		\begin{split}
			\rd_t \frac{1}{\bfrho^{(\epsilon)}} + (\bfu^{(\epsilon)} + \bff^{(\epsilon)} +  2\eps \frac{\nb \bfrho^{(\epsilon)}}{\bfrho^{(\epsilon)}} ) \cdot\nabla (\frac{1}{\bfrho^{(\epsilon)}})  =  (\nb \cdot \bff^{(\epsilon)}) (\frac{1}{\bfrho^{(\epsilon)}}) + \epsilon \lap \frac{1}{\bfrho^{(\epsilon)}} .
		\end{split}
	\end{equation*} Therefore, we obtain the estimate \begin{equation}\label{eq:rho-inf-viscous}
	\begin{split}
		\left| \frac{d}{dt} \nrm{\bfrho^{(\epsilon)}}_{inf}^{-1}  \right| \le \nrm{\lap\bfc^{(\epsilon)}}_{L^\infty}  \nrm{\bfrho^{(\epsilon)}}_{inf}^{-1} 
	\end{split}
\end{equation}
as in the proof of maximum principle for advection-diffusion equations. With a uniform--in--$\epsilon$ bound on $\nrm{\lap\bfc^{(\epsilon)}}_{L^\infty} $, this inequality guarantees that the quantity $\nrm{\bfrho^{(\epsilon)}}_{inf}^{-1} $ is actually uniformly bounded in some common time interval.

	Next, the uniform bound on the quantity $Z_m$ guarantees that for all $\eps>0$, the solution $(\bfrho^{(\epsilon)},\bff^{(\epsilon)},\bfc^{(\epsilon)})$ can be extended at least up to time interval $[0,T]$ for some $T>0$. By passing to a weakly convergent subsequence as $\epsilon>0$, one obtains a triple $(\bfrho,\bff,\bfc)$ with finite $Z_m$ on $[0,T]$. It is not difficult to show that this triple is a solution to \eqref{eq:KSF-inviscid}, with prescribed initial data. }
	
	\medskip
	
	\noindent \textit{(v) Uniqueness}: To prove uniqueness, we assume existence of two solutions $(\bfrho_i, \bff_i,\bfu_i,\bfp_i)$ ($i = 1, 2$) to \eqref{eq:KSF-inviscid} satisfying the properties stated in the theorem with the same initial data. We define \begin{equation*}
	\begin{split}
	\trho = \bfrho_1 - \bfrho_2, \quad \tc = \bfc_1-\bfc_2,\quad  \tf = \bff_1 - \bff_2, \quad \tu = \bfu_1 - \bfu_2, \quad \tilde{\bfp} = \bfp_1-\bfp_2.
	\end{split}
	\end{equation*} Note that $\trho$ and $\tu$ respectively satisfies \begin{equation*}
	\begin{split}
	\rd_t\trho + (\bfu_1+\bff_1) \cdot\nb\trho + (\tu+\tf)\cdot\nb\bfrho_2 = - \bfrho_2 \nb\cdot\tf - \trho \nb\cdot \bff_1 ,
	\end{split}
	\end{equation*} and \begin{equation*}
	\begin{split}
	\rd_t\tu +  \bfu_1 \cdot\nb\tu +  \tu \cdot\nb\bfu_2 + \nb\tilde{\bfp}=  \trho \nb\phi . 
	\end{split}
	\end{equation*} From the equations for $\trho$ and $\tu$, we obtain that \begin{equation}\label{eq:trho-modified}
	\begin{split}
	&\rd_t (\trho + \frac{\bff_2}{\bfc_2}\cdot \tu) + \bfu_1\cdot\nb (\trho + \frac{\bff_2}{\bfc_2}\cdot \tu) \\
	&\quad= -\bff_1\cdot\nb\trho - \bfrho_2\nb\cdot\tf - (\tu+\tf)\cdot\nb\bfrho_2 - \trho \nb\cdot \bff_1 - \tu \cdot\nb\bfu_2- \nb\tilde{\bfp} +  \trho \nb\phi \\
	&\quad\qquad + \rd_t( \frac{\bff_2}{\bfc_2} ) \cdot \tu + \bfu_1\cdot \nb ( \frac{\bff_2}{\bfc_2} ) \cdot \tu \\
	&\quad= -\bff_1\cdot \nb (\trho + \frac{\bff_2}{\bfc_2}\cdot \tu) - \bfrho_2 \nb\cdot (\tf - \frac{\bff_1}{\bfrho_2\bfc_2}\bff_2\cdot\tu)  - \bfrho_2 \nb\cdot (\frac{\bff_1}{\bfrho_2}) \frac{\bff_2}{\bfc_2}\cdot\tu  \\
	&\quad\qquad   - (\tu+\tf)\cdot\nb\bfrho_2 - \trho \nb\cdot \bff_1 - \tu \cdot\nb\bfu_2- \nb\tilde{\bfp} +  \trho \nb\phi + \rd_t( \frac{\bff_2}{\bfc_2} ) \cdot \tu + \bfu_1\cdot \nb ( \frac{\bff_2}{\bfc_2} ) \cdot \tu. 
	\end{split}
	\end{equation} On the other hand, $\tc$ and $\tf$ satisfy \begin{equation*}
	\begin{split}
	\rd_t\tc +  \bfu_1 \cdot\nb\tc +  \tu \cdot\nb\bfc_2 = -\bfrho_1\tc - \bfc_2 \trho ,
	\end{split}
	\end{equation*} and \begin{equation*}
	\begin{split}
	\rd_t\tf + \bfu_1 \cdot\nb\tf = -(\nb\bfu_1)^T \tf - \bfrho_1 \tf - \tc \nb \bfrho_1 - \bfc_2 \nb\trho - \trho\nb\bfc_2 - \tu\cdot\nb\bff_2 - (\nb\tu)^T \bff_2 .
	\end{split}
	\end{equation*} From the equation for $\tf$, we have \begin{equation}\label{eq:tf-modified}
	\begin{split}
	& \rd_t(\tf - \frac{\bff_1}{\bfrho_2\bfc_2}\bff_2\cdot\tu ) + \bfu_1 \cdot\nb(\tf - \frac{\bff_1}{\bfrho_2\bfc_2}\bff_2\cdot\tu ) \\
	& \quad = - \bfc_2  \nb (\bfrho + \frac{\bff_2}{\bfc_2}\tu) + \bfc_2 \nb\cdot (\frac{\bff_2}{\bfc_2}) \tu \\
	& \quad\qquad -\rd_t (  \frac{\bff_1}{\bfrho_2\bfc_2}\bff_2 ) \cdot\tu - \bfu_1\cdot\nb (  \frac{\bff_1}{\bfrho_2\bfc_2}\bff_2 ) \cdot \tu -(\nb\bfu_1)^T \tf - \bfrho_1 \tf - \tc \nb \bfrho_1 - \trho\nb\bfc_2 - \tu\cdot\nb\bff_2 . 
	\end{split}
	\end{equation} Now, from the equation for $\tu$, we easily obtain that \begin{equation*}
	\begin{split}
	\left|\frac{d}{dt} \nrm{\tu}_{L^2}^2\right| \lesssim \nrm{\tu}_{L^2}^2 +  \nrm{ \trho + \frac{\bff_2}{\bfc_2}\cdot \tu }_{L^2}^2. 
	\end{split}
	\end{equation*} {To derive this estimate, we have used the assumed regularity of the solutions  $(\bfrho_i, \bff_i,\bfu_i,\bfp_i)$ for $i = 1$ and $2$: $\tilde{\bfu} \in H^m$ ensures that $\rd_{t} \tilde{\bfu}$ is defined as a continuous function in space and therefore $\frac12\nrm{\tilde{\bfu}}_{L^2}^2$ is differentiable in time with time derivative equal to $\int \tilde{\bfu} \rd_t \tilde{\bfu}$.} Next, simultaneously using \eqref{eq:trho-modified} and \eqref{eq:tf-modified}, \begin{equation*}
	\begin{split}
\left|	\frac{d}{dt} \left(  \nrm{ \sqrt{\bfc_2} (\trho + \frac{\bff_2}{\bfc_2}\cdot \tu )}_{L^2}^2 +  \nrm{ \sqrt{\bfrho_2} (\tf - \frac{\bff_1}{\bfrho_2\bfc_2}\bff_2\cdot\tu )}_{L^2}^2 \right) \right|\lesssim \nrm{ \trho + \frac{\bff_2}{\bfc_2}\cdot \tu }_{L^2}^2 +  \nrm{\tf - \frac{\bff_1}{\bfrho_2\bfc_2}\bff_2\cdot\tu }_{L^2}^2 + \nrm{\tu}_{L^2}^2. 
	\end{split}
	\end{equation*} Hence, for \begin{equation*}
	\begin{split}
	\widetilde{X} :=\nrm{ \sqrt{\bfc_2} (\trho + \frac{\bff_2}{\bfc_2}\cdot \tu )}_{L^2}^2 +  \nrm{ \sqrt{\bfrho_2} (\tf - \frac{\bff_1}{\bfrho_2\bfc_2}\bff_2\cdot\tu )}_{L^2}^2 + \nrm{\tu}_{L^2}^2,
	\end{split}
	\end{equation*} we have \begin{equation*}
	\begin{split}
	\left|\frac{d}{dt} \widetilde{X} \right|\lesssim \widetilde{X}. 
	\end{split}
	\end{equation*} Since $\widetilde{X}(t=0) = 0$, we conclude that $\widetilde{X} = 0$ for $t>0$. This finishes the proof of uniqueness.

	\medskip
	
	\noindent \textit{(vi) Modification for general $\chi$ and $k$:} In the case of general $\chi$ and $k$, we simply define the good variables by \begin{equation*}
	\begin{split}
	R_m = \rd^m\bfrho + \frac{\nb\bfc}{k(\bfc)} \rd^m\bfu 
	\end{split}
	\end{equation*} and \begin{equation*}
	\begin{split}
	F_m = \rd^m \bff - \frac{\bff}{\bfrho k(\bfc)} \bff \cdot \rd^m\bfu .
	\end{split}
	\end{equation*} Moreover, instead of $\sqrt{\bfc}$, we use the weight  \begin{equation*}
	\begin{split}
	\left(  \frac{k(\bfc)}{\chi(\bfc)}\right)^{\frac{1}{2}}
	\end{split}
	\end{equation*} for $R_m$. The assumption \begin{equation*}
	\begin{split}
	  \frac{k(\bfc)}{\chi(\bfc)} \gtrsim \bfc^\gmm 
	\end{split}
	\end{equation*} allows us to bound \begin{equation*}
	\begin{split}
	\nrm{R_m}_{L^2} \lesssim \nrm{\bfc}_{inf}^{-\frac{\gmm}{2}} \nrm{ \left(  \frac{k(\bfc)}{\chi(\bfc)}\right)^{\frac{1}{2}} R_m }_{L^2}.
	\end{split}
	\end{equation*} We omit the details. The proof is now complete. 
\end{proof} 

{\begin{remark}[Extension to general sensitivity functions]\label{rem:rotation}
	In an interesting recent work \cite{WinklerIMRN}, Winkler considered the very general case where the chemotactic sensitivity function is given by a matrix $S$ depending on $x$, $\bfrho$, and $\bfc$. In the presence of diffusion, the author established eventual relaxation of Keller--Segel-fluid systems with general bounded and smooth  $S$. Here, let us briefly present an extension of Theorem \ref{thm:lwp-away} in the simplest case of $\chi(z) = 1$, $k(z) = z$, and $\bfu \equiv 0$: \begin{equation}\label{eq:rot}
		\left\{
		\begin{aligned}
		&\rd_t \boldsymbol{\rho} = -\nabla \cdot (\bfrho S \nabla  \bfc) , \\
		&\rd_t \bfc = -\bfc \bfrho . 
		\end{aligned}
		\right.
	\end{equation} Under the assumption that $\rho$ and $\bfc$ are bounded away from 0 at the initial time, we claim that the generalized system \eqref{eq:rot} is locally well-posed if the $n\times n$--matrix $S$ satisfies \begin{equation*}
	\begin{split}
		S + S^T \mbox{ is a diagonal matrix with strictly positive diagonal entries}. 
	\end{split}
\end{equation*} That is, $S = (s_{ij})_{1\le i,j \le n}$ satisfies $s_{ij}+s_{ji}=0$ for $i \ne j$ and $s_{jj}>0$. When $S$ is dependent upon $x, \bfrho, \bfc$, we just need to require the above condition (together with uniform boundedness and smoothness of $S$) uniformly for each $S(x,\bfrho,\bfc)$. In the case of two spatial dimensions $n=2$, this means that our local well-posedness theory can cover rotation matrices with rotation angle strictly less than $\pi/2$.\footnote{The case of rotation with angle $\pi/2$ results in a very interesting system, for which the question of local well/ill-posedness seems delicate.} To prove this extension, one can notice cancellations between the top order terms in the expression \begin{equation*}
\begin{split}
	\frac{d}{dt}  \frac12 \int  \bfc (\rd^m\bfrho)^2 + \sum_{j=1}^{n}  s_{jj}\bfrho |\rd^{m}\rd_{j}\bfc|^2. 
\end{split}
\end{equation*}  For $m$ large, the top order terms in the above expression are \begin{equation*}
\begin{split}
	I = \int \bfc \bfrho  \nb (\rd^m\bfrho) \cdot (S \nb (\rd^m\bfc)), \qquad  II=- \sum_{j=1}^{n} \int \bfc \bfrho  \rd_j (\rd^m\bfc) \rd_j (\rd^m\bfrho). 
\end{split}
\end{equation*} Then, in the expression $I$, the contributions from off-diagonal terms of $S$ cancel each other due to the condition $s_{ij} + s_{ji} = 0$ for $i\ne j$, and the contribution from the diagonal entries of $S$ cancel precisely with $II$. 
\end{remark}}

\begin{remark}[Ill-posedness in the case of opposite signs] \label{rem:illposed}
	We note that having the same signs in the right hand side of \eqref{eq:KS-1D} is crucial for well-posedness. Indeed, consider the following inviscid system obtained by reverting the sign of the right hand side for $\rd_t\bfc$: 
	\begin{equation}\label{eq:KS-1D-prime}
		\left\{
		\begin{aligned}
			\rd_t \bfrho & =  -\rd_x ( \bfrho \rd_x \bfc), \\
			\rd_t \bfc & =  \bfc\bfrho  .
		\end{aligned}
		\right.
	\end{equation} Similarly as before, $(\bfrho,\bfc) = (1, e^{t})$ provides a solution to the above. The linearization around this solution is given by \begin{equation}\label{eq:KS-1D-prime-lin}
		\left\{
		\begin{aligned}
			\rd_t \tilde{\bfrho} & =  -\rd_{xx}\tilde{\bfc}, \\
			\rd_t \tilde{\bfc} & =  e^{t}\tilde{\bfrho} + \tilde{\bfc}.
		\end{aligned}
		\right.
	\end{equation} It is not difficult to show that this linear equation is \textit{ill-posed}: there exist smooth data which immediately lose smoothness for $t>0$. Indeed, taking another time derivative, \begin{equation*}
		\begin{split}
			\rd_{tt}\tilde{\bfrho} = -e^{t}\rd_{xx}\tilde{\bfrho} - \rd_{xx}\tilde{\bfc}
		\end{split}
	\end{equation*} Modulo the coefficient $e^{t}$ and the lower order term $\rd_{xx}\tilde{\bfc}$, we have the Laplace equation in the $(t,x)$-coordinates, whose initial value problem is well-known to be illposed. A similar analysis can be repeated for the system \begin{equation}\label{eq:KS-1D-prime2}
	\left\{
	\begin{aligned}
		\rd_t \bfrho & =  -\rd_x ( \bfrho \rd_x \bfc), \\
		\rd_t \bfc & =  \bfrho  .
	\end{aligned}
	\right.
\end{equation}
\end{remark}

\subsection{Well-posedness with high order vanishing data}\label{subsec:v1}

Let us begin this section with stating a version of Theorem \ref{MainThm3}: for this purpose, we need to introduce 
\begin{definition}
	We introduce the following function spaces which is defined only for non-negative functions. We say that $g\ge 0$ belongs to $\dot{\calX}^{m,\gmm}$ for some $m\ge 1$ and $0<\gmm<2$ if {$g \in H^{m}$} and the quantity \begin{equation}\label{eq:Xmr}
	\begin{split}
	\nrm{g}_{\dot{\calX}^{m,\gmm}}^{2} := \nrm{g^{-\frac{\gmm}{2}} |\nb^mg| }_{L^2}^2 + \nrm{g^{-(1+\frac{\gmm-2}{2m})} \nb g }_{L^{2m}}^{2m} := \int \frac{|\nb^m g|^2}{g^{\gmm}} + \int \frac{|\nb g|^{2m}}{g^{2m+\gmm-2}}
	\end{split}
	\end{equation} is finite. {To be precise, the integrals in the right hand side are defined by
\begin{equation*}
	\begin{split}
		\lim_{\eps\to 0^+}  \left[ \int \frac{|\nb^m g|^2}{(g+\eps)^{\gmm}} + \int \frac{|\nb g|^{2m}}{(g+\eps)^{2m+\gmm-2}} \right] \, .
	\end{split}
\end{equation*} For $g \in H^{m}$, the above quantities are well-defined by allowing $+\infty$.}  In the case of $m = 0$, we simply set \begin{equation*}
	\begin{split}
	\nrm{g}_{\dot{\calX}^{0,\gmm}}^{2} := \int g^{2-\gmm},
	\end{split}
	\end{equation*} and then we define \begin{equation*}
	\begin{split}
	\nrm{g}_{ {\calX}^{m,\gmm}}^{2} := \sum_{k=0}^{m} \nrm{g}_{\dot{\calX}^{k,\gmm}}^{2} .
	\end{split}
	\end{equation*}
\end{definition}
We are now ready to precisely state the local well-posedness result for possibly vanishing data. 
\begin{theorem}[Well-posedness with vanishing data]\label{thm:lwp-rho-away}
	We consider the fully inviscid \eqref{eq:KS} system ($D_\rho, D_c = 0$) on $\bbT^d$ under the following assumptions on the coefficients:
	\begin{itemize}
		\item $\chi$ is $C^\infty$--smooth and uniformly positive,
		\item $\phi \in C^\infty(\bbT^d)$, and
		\item there exist  a $C^\infty$--smooth function $F_k : [0,\infty) \rightarrow (0,\infty)$ and $0<\gmm<2$ such that $k(z) = F_k(z^{\gmm})$.
	\end{itemize}
	Furthermore, we impose the following assumptions on the data: for some $2 + \frac{d}{2} < m $, 
	\begin{itemize}
		\item $\bfc_0 \in {{\calX}^{m+1,\gmm}} \cap L^\infty$, 
		\item $\bfrho_0 \in \calX^{m,1}  \cap L^\infty $, and
		\item there exists some $\dlt$ satisfying \begin{equation*}
		\begin{split}
		1-\gmm \le \dlt \le \min\left\{ \gmm, 2\left(m-1- \lfloor \frac{d}{2} \rfloor \right)\left( 1 - \frac{1-\frac{\gmm}{2}}{ m - \lfloor \frac{d}{2} \rfloor  } \right) \right\} 
		\end{split}
		\end{equation*} such that \begin{equation*}
		\begin{split}
		a \le \frac{\bfrho_0(x)}{\bfc_0^{\dlt}(x)} \le A 
		\end{split}
		\end{equation*} for some constants $A,a>0$.
	\end{itemize} Then the system \eqref{eq:KS} is locally-well posed: there exist some $T > 0$ and a unique solution $(\bfrho,\bfc)$ to \eqref{eq:KS} satisfying the initial condition and \begin{equation*}
	\begin{split}
	\bfrho\in L^\infty([0,T);\calX^{m,1} \cap L^\infty), \quad \bfc  \in L^\infty([0,T); {{\calX}^{m+1,\gmm}} \cap L^\infty), \quad \frac{\bfrho(x)}{\bfc^{\dlt}(x)} +\frac{\bfc^\dlt(x)}{\bfrho(x)} \in L^\infty([0,T);L^\infty).  
	\end{split}
	\end{equation*}
\end{theorem}
\begin{remark}[Examples of initial data]
	When a function is bounded away from 0, the $\calX^{m,\gmm}$-norm is simply equivalent with a usual Sobolev norm. In the case when the function has an isolated zero (say the origin), one can ensure that it belongs to $\calX^{m,\gmm}$ as long as the order of vanishing is high enough. To see this, say $g(x) \simeq |x|^{n}$ for some $n\ge 0$ in $\bbT^d$. Then, we formally compute that, near $x = 0$, \begin{equation*}
	\begin{split}
	\int \frac{|\nb^m g|^2}{g^{\gmm}} \simeq \int |x|^{ (2-\gmm)n - 2m}, \quad \int \frac{|\nb g|^{2m}}{g^{2m+\gmm-2}} \simeq \int |x|^{ (2-\gmm)n - 2m }. 
	\end{split}
	\end{equation*} Therefore, locally at $x = 0$,  $g \in \calX^{m,\gmm}$ if and only if \begin{equation*}
	\begin{split}
	n > \frac{2m-d}{2-\gmm}. 
	\end{split}
	\end{equation*} (This condition actually forces the condition $\gmm<2$ in the statement of Theorem \ref{thm:lwp-rho-away}.)
\end{remark}

We divide the proof into three parts. First, in \ref{subsubsec:GNS}, we establish some weighted inequalities of Gagliardo--Nirenberg--Sobolev type. We prove a priori estimates for the solution in \ref{subsubsec:apriori}. Finally, we show existence and uniqueness of the solution satisfying the a priori estimates in \ref{subsubsec:eandu}.

\subsubsection{A chain of weighted Gagliardo--Nirenberg--Sobolev inequalities}\label{subsubsec:GNS}

In this section, we shall state a series of inequalities which generalize the well known Gagliardo--Nirenberg--Sobolev inequality to the weighted case. We briefly remark on the notation used: Given an integer $m\ge 0$, we denote $\nb^m g$ to denote the vector consisting of all partial derivatives for $g$ of order $m$. On the other hand, given some $d$-vector $\ell = (\ell^1,\cdots, \ell^d)$ with integer $\ell^i\ge0$, we define $\rd^\ell g = \rd_{x_1}^{(\ell^1)} \cdots \rd_{x_d}^{(\ell^d)}g$. In particular, $\rd^\ell g$ is an element of $\nb^{|\ell|} g$, where $|\ell| = \ell^1 + \cdots + \ell^d$. With the above notation, we have the following key lemma:
\begin{lemma}\label{lem:inequality}
	Let $\bfc \ge 0$ belong to $H^{m}(\bbT^d)$. For any integer $m\ge 1$ and a $k$-tuple $(\ell_1, \cdots, \ell_k)$ of $d$-vectors satisfying $m = |\ell_1| + \cdots + |\ell_k|$, we have that\begin{equation}\label{eq:inequality}
	\begin{split}
	\int_{ \{ \bfc > 0 \} } \frac{ 1 }{ \bfc^{2k-1} } \prod_{1\le i \le k} |\rd^{\ell_i}\bfc|^2 \lesssim_{m,d} \int_{ \{ \bfc > 0 \} } \frac{|\nb^m\bfc|^2}{\bfc} + \int_{ \{ \bfc > 0 \} } \frac{|\nb\bfc|^{2m}}{\bfc^{2m-1}}
	\end{split}
	\end{equation} {as long as the expression on the right hand side is finite.} In particular, under the same assumptions, we obtain that for any $0\le n \le m$, 
	\begin{equation}\label{eq:inequality-conseq}
	\begin{split}
	\nrm{ \nb^{m-n}( \frac{\nb^{n}\bfc}{\sqrt{\bfc} })}_{L^2({ \{ \bfc > 0 \} })}^2 \lesssim_{m,d}  \int_{ \{ \bfc > 0 \} } \frac{|\nb^m\bfc|^2}{\bfc} + \int_{ \{ \bfc > 0 \} } \frac{|\nb\bfc|^{2m}}{\bfc^{2m-1}} . 
	\end{split}
	\end{equation}
\end{lemma} 
As a corollary, we obtain the following:
\begin{corollary}\label{cor:L-infty}
	Assume that for some integer $k\ge 1$, $\bfc\ge 0$ satisfies \begin{equation*}
	\begin{split}
	\int_{ \{ \bfc > 0 \} } \frac{|\nb^m\bfc|^2}{\bfc} + \int_{ \{ \bfc > 0 \} } \frac{|\nb\bfc|^{2m}}{\bfc^{2m-1}}  < + \infty
	\end{split}
	\end{equation*} for $m = 1, \cdots,  k + 1 + \lfloor \frac{d}{2} \rfloor $. Then
	\begin{equation}\label{eq:inequality-infty}
	\begin{split}
	\sup_{x : \bfc(x) >0 } \frac{|\nb\bfc(x)|^{k}}{\bfc(x)^{k-\frac{1}{2}}} \le C_d\sum_{m=1}^{\lfloor \frac{d}{2} \rfloor +1 + k} \left( \int_{ \{ \bfc > 0 \} } \frac{|\nb^m\bfc|^2}{\bfc} + \int_{ \{ \bfc > 0 \} } \frac{|\nb\bfc|^{2m}}{\bfc^{2m-1}}  \right) .
	\end{split}
	\end{equation}
\end{corollary} 
\begin{proof}
	The proof is immediate from Lemma \ref{lem:inequality} and the embedding $H^s(\bbT^d) \subset L^\infty(\bbT^d)$ for $s> \frac{d}{2}$. Indeed, we have that \begin{equation*}
	\begin{split}
	 \frac{|\nb\bfc(x)|^{k}}{\bfc(x)^{k-\frac{1}{2}}} \in L^2({ \{ \bfc > 0 \} }), 
	\end{split}
	\end{equation*} and for all $s \le 1 + \lfloor \frac{d}{2} \rfloor $, \eqref{eq:inequality} applies to each term in the expression \begin{equation*}
	\begin{split}
	\nb^{s} \left(   \frac{|\nb\bfc(x)|^{k}}{\bfc(x)^{k-\frac{1}{2}}} \right) 
	\end{split}
	\end{equation*} to give an $L^2$ bound, with $m = k + s$. \end{proof}

\begin{proof}[{\bf{Proof of Lemma \ref{lem:inequality}}}]
	The proof is based on first establishing
	\begin{equation}\label{eq:ineq-GNS}
	\begin{split}
	\int_{ \{ \bfc > 0 \} } \frac{|\nb^\ell \bfc|^{\frac{2m}{\ell}}}{\bfc^{\frac{2m}{\ell}-1}} \lesssim_{m,d} \int_{ \{ \bfc > 0 \} } \frac{|\nb^m\bfc|^2}{\bfc} + \int_{ \{ \bfc > 0 \} } \frac{|\nb\bfc|^{2m}}{\bfc^{2m-1}} 
	\end{split}
	\end{equation} for every $1 < \ell < m$, {assuming that $\bfc\ge0$ is a given function with finite right-hand side. In the arguments below, we shall actually assume that $\bfc>0$. This is possible since we may first prove \eqref{eq:ineq-GNS} with $\bfc$ replaced by $\bfc+\eps$ and then take the limit $\eps\to0$ first on the right-hand side (using the finiteness assumption for $\bfc$ given in Lemma \ref{lem:inequality}). After that, we may take the limit $\eps\to0$ on the left-hand side.} 

	Note that each term in \eqref{eq:ineq-GNS} has $2m$ derivatives of $\bfc$ and 1-homogeneity in $\bfc$. Once we have this inequality, we immediately deduce \begin{equation*}
	\begin{split}
	\int \frac{ 1 }{ \bfc^{2k-1} } \prod_{1\le i \le k} |\rd^{\ell_i}\bfc|^2\le \prod_{1\le i\le k} \left(\int \frac{|\nb^{|\ell_i|} \bfc|^{\frac{2m}{|\ell_i|}}}{\bfc^{\frac{2m}{|\ell_i|}-1}}\right)^{\frac{|\ell_i|}{m}} \lesssim_{m,d} \int \frac{|\nb^m\bfc|^2}{\bfc} + \int \frac{|\nb\bfc|^{2m}}{\bfc^{2m-1}}  
	\end{split}
	\end{equation*} from H\"older's inequality. 
	
	\medskip
	
	\noindent {Since there is nothing to prove in the case $m=2$, we shall verify \eqref{eq:ineq-GNS} for some $m>2$. We need to prove that} \begin{equation}\label{eq:goal}
		\begin{split}
			I_{\ell} \lesssim_{m,d} {I_1} + I_{m}, \quad 1<\ell< m
		\end{split}
	\end{equation}  with \begin{equation*}
		\begin{split}
			I_{\ell} := \int \frac{|\nb^\ell  \bfc|^{\frac{{2m}}{\ell}}}{ \bfc^{\frac{{2m}}{\ell}-1}} .
		\end{split}
	\end{equation*} For a partial derivative $\rd^\ell$ of order $\ell$,  \begin{equation*}
		\begin{split}
			&\int \frac{|\rd^\ell \bfc|^{\frac{2m}{\ell}-1}}{\bfc^{{\frac{2m}{\ell}-1}}} \mathrm{sgn} (\rd^{\ell}\bfc)  \rd^\ell \bfc   \\
			&\quad = - ({\frac{2m}{\ell}-1})\int \frac{|\rd^\ell \bfc|^{\frac{2m}{\ell}-2}}{\bfc^{\frac{2m}{\ell}-1}} \, \rd^{\ell-1} \bfc \, \rd^{\ell+1} \bfc + ({\frac{2m}{\ell}-1} )\int \frac{|\rd^\ell \bfc|^{\frac{2m}{\ell}-1}}{\bfc^{\frac{2m}{\ell}}} \, \rd \bfc \, \rd^{\ell-1} \bfc \, \mathrm{sgn}(\rd^\ell \bfc). 
		\end{split}
	\end{equation*} Note that $2<\frac{2m}{\ell} \le m$. Applying  H\"older's inequality to the right hand side, we have \begin{equation*}
		\begin{split}
			I_{\ell} \lesssim_{\ell,m} I_{\ell-1}^{ \frac{\ell-1}{2m}} I_{\ell}^{1-\frac{\ell}{m}} I_{\ell+1}^{\frac{\ell+1}{2m}} + I_1^{\frac{1}{2m}} I_{\ell-1}^{\frac{\ell-1}{2m}} I_{\ell}^{1-\frac{\ell}{2m}}. 
		\end{split}
	\end{equation*}  With $\epsilon$-Young inequality, we obtain \begin{equation*}
		\begin{split}
			I_{\ell} \le \eps I_\ell + \eps I_{\ell-1} + C_{\eps,\ell,m} ( I_1 +   I_{\ell+1} ). 
		\end{split}
	\end{equation*} After a suitable weighted summation of the inequality in all of the cases $\ell = 2, \cdots, m-1$, one can obtain \eqref{eq:goal}. 
\end{proof}
At this point, one may note that Lemma \ref{lem:key}  is a straightforward generalization of the above. The proof can be done in a completely parallel manner with Lemma \ref{lem:inequality}.

{
	\begin{remark}
		In the special case $m=2$, we have after an integration by parts 
		\[
		\int\frac{{\abs{\nabla \bfc}}^4}{\bfc^3}=-\int\frac{\abs{\nabla \bfc}^2\Delta \bfc}{\bfc^2}+3\int\frac{{\abs{\nabla \bfc}}^4}{\bfc^3},
		\]
		which implies that
		\[
		2\int\frac{{\abs{\nabla \bfc}}^4}{\bfc^3}\le \int\frac{\abs{\nabla \bfc}^2\Delta \bfc}{\bfc^2}
		\le \int\frac{{\abs{\nabla \bfc}}^4}{\bfc^3}+\frac{1}{4}\int \frac{\abs{\Delta \bfc}^2}{\bfc}
		\]
		by H\"older's inequality. Therefore, the first term on the right-hand side of \eqref{eq:ineq-GNS} is controlled by the other term. Unfortunately, this does not seem to hold for $m>2$. 
	\end{remark}

}

\subsubsection{A priori estimates}\label{subsubsec:apriori}

{Let us proceed to obtain a priori estimates for the solution, assuming that $\bfc, \bfrho$ are sufficiently smooth and non-negative.} Furthermore, we take the simplifying assumption that $k(\bfc) = \bfc^\gmm$ and $\chi(\bfc) = 1$. Extending the proof to the case of general $k$ and $\chi$ as stated in Theorem \ref{thm:lwp-rho-away} is straightforward. 

\medskip

\noindent \textit{(i) Ratio estimates}: We would like to propagate $\bfrho\lesssim \bfc^\dlt $ and $\bfc^\dlt \lesssim \bfrho$. To this end, \begin{equation*}
\begin{split}
\frac{\partial}{\partial t}  \frac{\bfc^\dlt}{\bfrho} = \frac{\bfc^\dlt}{\bfrho}\frac{ \nb\cdot(\chi(\bfc)\bfrho\nb\bfc)}{\bfrho} - \dlt k(\bfc) \bfc^{\dlt-1} 
\end{split}
\end{equation*} so that recalling $\chi = 1$ and $k(\bfc)=\bfc^{\gmm}$, 
\begin{equation*}
\begin{split}
{\frac{d}{dt}} \left(\nrm{ \frac{\bfc^\dlt}{\bfrho}}_{L^\infty} + \nrm{ \frac{\bfrho}{\bfc^\dlt}}_{L^\infty}\right) \le C\nrm{  \frac{ \nb\cdot(\bfrho\nb\bfc)}{\bfrho}}_{L^\infty} \left(\nrm{ \frac{\bfc^\dlt}{\bfrho}}_{L^\infty} + \nrm{ \frac{\bfrho}{\bfc^\dlt}}_{L^\infty}\right)  + C\nrm{\bfc^{\dlt+\gmm-1}}_{L^\infty}. 
\end{split}
\end{equation*} From the assumption $\dlt+\gmm-1\ge 0$, we have that \begin{equation*}
\begin{split}
\nrm{\bfc^{\dlt+\gmm-1}}_{L^\infty}\le \nrm{\bfc}_{L^\infty}^{\dlt+\gmm-1}
\end{split}
\end{equation*} and \begin{equation*}
\begin{split}
\nrm{  \frac{ \nb\cdot(\bfrho\nb\bfc)}{\bfrho}}_{L^\infty} \le C( \nrm{\lap\bfc}_{L^\infty} + \nrm{ \frac{\bfc^\dlt}{\bfrho}}_{L^\infty}^a \nrm{\frac{\nb\bfrho}{\bfrho^{1-a}}}_{L^\infty} \nrm{\frac{\nb\bfc}{\bfc^{\dlt a}}}_{L^\infty} ),
\end{split}
\end{equation*} where $0<a<1$ is to be chosen below. From \eqref{eq:inequality-infty} in Corollary \ref{cor:L-infty}, we have that \begin{equation*}
\begin{split}
\nrm{\frac{\nb\bfrho}{\bfrho^{1-a}}}_{L^\infty} \le C \nrm{\bfrho}_{{\calX}^{m,1}}^{2a}, \quad a := \frac{1}{2(m-1-\lfloor \frac{d}{2} \rfloor)}. 
\end{split}
\end{equation*} Hence, we have \begin{equation*}
\begin{split}
 \nrm{\frac{\nb\bfc}{\bfc^{\dlt a}}}_{L^\infty} \le C \nrm{\bfc}_{ {\calX}^{m+1,\gmm} }^{\frac{1}{m- \lfloor \frac{d}{2} \rfloor }} \nrm{\bfc}_{L^\infty}^{\eps}
\end{split}
\end{equation*} for some $\eps\ge0$ if \begin{equation*}
\begin{split}
\dlt\le  2\left(m-1- \lfloor \frac{d}{2} \rfloor \right)\left( 1 - \frac{1-\frac{\gmm}{2}}{ m - \lfloor \frac{d}{2} \rfloor  } \right). 
\end{split}
\end{equation*} This, together with $ \nrm{\lap\bfc}_{L^\infty} \lesssim \nrm{c}_{L^\infty}^{\frac{\gmm}{2}} \nrm{\bfc}_{ {\calX}^{m+1,\gmm} } $, we have \begin{equation}\label{eq:L-infty}
\begin{split}
\nrm{  \frac{ \nb\cdot(\bfrho\nb\bfc)}{\bfrho}}_{L^\infty} \le C(1 + \nrm{\bfc}_{L^\infty})^b  \left(1 + \nrm{ \frac{\bfc^\dlt}{\bfrho}}_{L^\infty}\right)^{a} (1 + \nrm{\bfc}_{ {\calX}^{m+1,\gmm} } + \nrm{\bfrho}_{ {\calX}^{m,1} } )^2 
\end{split}
\end{equation} and then \begin{equation}\label{eq:apriori-1}
\begin{split}
\frac{d}{dt}  \left(\nrm{ \frac{\bfc^\dlt}{\bfrho}}_{L^\infty} + \nrm{ \frac{\bfrho}{\bfc^\dlt}}_{L^\infty}\right) \le C \left(1 + \nrm{ \frac{\bfc^\dlt}{\bfrho}}_{L^\infty} + \nrm{ \frac{\bfrho}{\bfc^\dlt}}_{L^\infty}\right)^{1 + a} (1 + \nrm{\bfc}_{L^\infty})^b (1 + \nrm{\bfc}_{ {\calX}^{m+1,\gmm} } + \nrm{\bfrho}_{ {\calX}^{m,1} } )^2 
\end{split}
\end{equation} follows, with some $b = b(\gmm,\dlt,d)>0$. Before we proceed, we note the trivial estimate \begin{equation}\label{eq:apriori-2}
\begin{split}
\frac{d}{dt} \left( \nrm{\bfc}_{L^\infty} + \nrm{\bfrho}_{L^\infty}  \right) \le C \nrm{\lap\bfc}_{L^\infty}\nrm{\bfrho}_{L^\infty} + C\nrm{\nb\bfrho}_{L^\infty}\nrm{\nb\bfc}_{L^\infty} \le C \nrm{\bfc}_{\calX^{m+1,\gmm}} (\nrm{\bfrho}_{L^\infty} + \nrm{\bfrho}_{\calX^{m,1}}  ),
\end{split}
\end{equation} since $\bfc$ is pointwise decreasing with time.
\medskip

\noindent \textit{(ii) $\calX^m$--estimates}:
We begin with the system \begin{equation}  \label{eq:KS-1st}
\left\{
\begin{aligned} 
\rd_t\bff & = -k'(\bfc)\bfrho\bff - k(\bfc)\nb\bfrho ,\\
\rd_t\bfrho & = - \chi(\bfc)\bff \cdot\nb\bfrho - \chi(\bfc)\bfrho \nb\cdot\bff - \chi'(\bfc)\bfrho |\bff|^2 . 
\end{aligned}
\right.
\end{equation} We write 
\begin{equation*}
\begin{split}
\frac{d}{dt} \frac{1}{2} \int \frac{\chi(\bfc)}{k(\bfc)} |\rd^m\bff|^2 + \frac{1}{\bfrho} |\rd^m\bfrho|^2 = I + II + III + IV, 
\end{split}
\end{equation*} with \begin{equation*}
\begin{split}
I = \frac{1}{2} \int\rd_t (  \frac{\chi(\bfc)}{k(\bfc)}  )|\rd^m\bff|^2 ,
\end{split}
\end{equation*}\begin{equation*}
\begin{split}
II = \frac{1}{2} \int \rd_t( \frac{1}{\bfrho} )|\rd^m\bfrho|^2 ,
\end{split}
\end{equation*}
\begin{equation*}
\begin{split}
III = \int   \frac{\chi(\bfc)}{k(\bfc)} \rd^m\bff \cdot \rd^m \left( -k'(\bfc)\bfrho\bff - k(\bfc)\nb\bfrho  \right),
\end{split}
\end{equation*} and \begin{equation*}
\begin{split}
IV = \int \frac{1}{\bfrho}\rd^m\bfrho \rd^m\left(  - \chi(\bfc)\bff \cdot\nb\bfrho - \chi(\bfc)\bfrho \nb\cdot\bff - \chi'(\bfc)\bfrho |\bff|^2 \right). 
\end{split}
\end{equation*} 
We estimate \begin{equation*}
\begin{split}
|I| = \frac{\gmm}{2} \int \bfc^{\dlt+\gmm-1} \frac{\bfrho}{\bfc^{\dlt}}  \frac{1}{\bfc^{\gmm}} |\rd^m\bff|^2 \le C \nrm{\bfc}_{L^\infty}^{\dlt+\gmm-1} \nrm{ \frac{\bfrho}{\bfc^{\dlt}} }_{L^\infty} \nrm{\bfc}_{\calX^{m+1,\gmm}}
\end{split}
\end{equation*} and \begin{equation*}
\begin{split}
|II| =\left| \frac{1}{2} \int \frac{\nb\cdot(\bfrho\nb\bfc)}{\bfrho} \frac{|\rd^m\bfrho|^2}{\bfrho} \right| \le  C(1 + \nrm{\bfc}_{L^\infty})^b   (1 + \nrm{ \frac{\bfc^\dlt}{\bfrho}}_{L^\infty} )^{a} (1 + \nrm{\bfc}_{ {\calX}^{m+1,\gmm} } + \nrm{\bfrho}_{ {\calX}^{m,1} } )^2 \nrm{\bfrho}_{ {\calX}^{m,1} }
\end{split}
\end{equation*} recalling $\chi = 1$, $k(z)=z^\gmm$ and using \eqref{eq:L-infty}. It remains to handle $III$ and $IV$. In the case of $IV$, it suffices to estimate integrals of the form \begin{equation*}
\begin{split}
\int \frac{1}{\bfrho} \rd^m\bfrho \rd^{\ell} \bff \rd^{m-\ell+1}\bfrho,
\end{split}
\end{equation*} where $0\le \ell \le m+1$. Let us first consider the extreme cases, which are the most troublesome. When $\ell = 0$, with an integration by parts, we have that the corresponding term in $IV$ is given by \begin{equation*}
\begin{split}
\int - \frac{\bff}{\bfrho} \cdot \nb \rd^m \bfrho \,\rd^m\bfrho = \frac{1}{2}\int  \left( \lap\bfc - \frac{\bff\cdot \nb\bfrho }{\bfrho}  \right)  \frac{|\rd^m\bfrho|^2}{\bfrho}. 
\end{split}
\end{equation*}  Recalling the estimate \eqref{eq:L-infty}, we have \begin{equation*}
\begin{split}
\left| \int - \frac{\bff}{\bfrho} \cdot \nb \rd^m \bfrho \,\rd^m\bfrho \right| \le C (\nrm{\frac{\bfc^{\dlt}}{\bfrho}}_{L^\infty}\nrm{\bfc}_{L^\infty}^{\gmm-\dlt} +1 )\nrm{\bfc}_{\calX^{m+1,\gmm}} \nrm{\bfrho}_{\calX^{m,1}}. 
\end{split}
\end{equation*} On the other hand, when $\ell = m+1$, we are left with \begin{equation*}
\begin{split}
-\int \rd^m\bfrho \nb\cdot \rd^m\bff.
\end{split}
\end{equation*} The case $0<\ell<m+1$ is easily bounded by Gagliardo--Nirenberg--Sobolev inequality: we take \begin{equation*}
\begin{split}
\left| \int \frac{1}{\bfrho} \rd^m\bfrho \rd^{\ell} \bff \rd^{m-\ell+1}\bfrho \right| \le C\nrm{\bfrho}_{\calX^{m,1}} \nrm{ \rd^{\ell} \bff  }_{L^{2p^*}} \nrm{ \bfrho^{-\frac{1}{2}} \rd^{m-\ell+1}\bfrho  }_{L^{2p}}
\end{split}
\end{equation*} where \begin{equation*}
\begin{split}
p = \frac{m}{m-\ell+1}, \quad \frac{1}{p} + \frac{1}{p^*} = 1. 
\end{split}
\end{equation*} Then, recalling the proof of Lemma \ref{lem:inequality}, we have \begin{equation*}
\begin{split}
\left(\int \frac{(\rd^{m-\ell+1}\bfrho)^{2p}}{\bfrho^{2p(\frac{1}{2} + \frac{\ell-1}{2m})}} \right)^{\frac{1}{2p}}\le C \nrm{\bfrho}_{\calX^{m,1}}^{\frac{1}{p}}
\end{split} 
\end{equation*} so that \begin{equation*}
\begin{split}
\nrm{ \bfrho^{-\frac{1}{2}} \rd^{m-\ell+1}\bfrho  }_{L^{2p}} \le C \nrm{\bfrho}_{L^\infty}^{\frac{\ell-1}{2m}} \nrm{\bfrho}_{\calX^{m,1}}^{\frac{1}{p}}. 
\end{split}
\end{equation*} Similarly, we can estimate \begin{equation*}
\begin{split}
\nrm{\rd^{\ell}\bff}_{L^{2p^*}} \le C \nrm{\bfc}_{L^\infty}^{b} \nrm{\bfc}_{\calX^{m+1,\gmm}}^{\frac{1}{p^*}}
\end{split}
\end{equation*} with some $b = b(\gmm,m,\ell)>0$. Therefore, \begin{equation*}
\begin{split}
\left| IV + \int \rd^m\bfrho \nb\cdot \rd^m\bff \right| \le   C(1 + \nrm{\bfc}_{L^\infty} + \nrm{\bfrho}_{L^\infty})^b (1+\nrm{\frac{\bfc^{\dlt}}{\bfrho}}_{L^\infty} )(1+\nrm{\bfc}_{\calX^{m+1,\gmm}})(1+ \nrm{\bfrho}_{\calX^{m,1}})\nrm{\bfrho}_{\calX^{m,1}}. 
\end{split}
\end{equation*} with some $b = b(\gmm,m)>0$. Estimating $III$ is similar. Note that \begin{equation*}
\begin{split}
III = -\int \frac{1}{\bfc^{\gmm}} \rd^m\bff \cdot \rd^m \nb (\bfc^\gmm\bfrho)
\end{split}
\end{equation*} and we need to consider \begin{equation*}
\begin{split}
\int \frac{1}{\bfc^{\gmm}} \rd^m\bff \rd^{m+1-\ell} ( \bfc^\gmm  )\rd^{\ell}\bfrho 
\end{split}
\end{equation*} for $0\le\ell\le m+1$. In the case $\ell = m+1$, we simply have \begin{equation*}
\begin{split}
-\int  \rd^m\bff \cdot \rd^m \nb \bfrho. 
\end{split}
\end{equation*} Next, when $\ell \le m$, we first rewrite the corresponding integral as \begin{equation*}
\begin{split}
- \int \frac{\rd^{\ell}\bfrho}{\bfc^{\gmm}} \rd^m \bff \cdot \rd^{m-\ell} (\gmm\bfc^{\gmm-1}\bff). 
\end{split}
\end{equation*} We show how to handle the case $\ell = 0$. Extending the estimate to the case $0<\ell \le m$ is simpler. We then need to consider \begin{equation*}
\begin{split}
-\gmm \int \frac{\bfrho}{\bfc^{\gmm}} \rd^m\bff \cdot  \rd^{j} ( \bfc^{\gmm-1} ) \rd^{m-j}\bff =  -\gmm \int \frac{\bfrho}{\bfc^{\dlt}} \bfc^{\dlt+\gmm-1} \frac{\rd^m\bff}{\bfc^{\frac{\gmm}{2}}} \cdot \frac{\rd^{j} ( \bfc^{\gmm-1} )}{\bfc^{\gmm-1 + \frac{1}{p^*} (\frac{\gmm}{2}-1)}} \frac{\rd^{m-j}\bff }{ \bfc^{1-\frac{1}{p} + \frac{\gmm}{2p}} } . 
\end{split}
\end{equation*} In the above rewriting, we choose \begin{equation*}
\begin{split}
p = \frac{m+1}{m+1-j} >1, \quad \frac{1}{p}+\frac{1}{p^*} = 1. 
\end{split}
\end{equation*} Then, \begin{equation*}
\begin{split}
\left| -\gmm \int \frac{\bfrho}{\bfc^{\gmm}} \rd^m\bff \cdot  \rd^{j} ( \bfc^{\gmm-1} ) \rd^{m-j}\bff \right| & \le C \nrm{\frac{\bfrho}{\bfc^{\dlt}}}_{L^\infty} \nrm{\bfc}_{L^\infty}^{\dlt+\gmm-1} \nrm{ \frac{\rd^{j} ( \bfc^{\gmm-1} )}{\bfc^{\gmm-1 + \frac{1}{p^*} (\frac{\gmm}{2}-1)}} }_{L^{2p^*}} \nrm{ \frac{\rd^{m-j}\bff }{ \bfc^{1-\frac{1}{p} + \frac{\gmm}{2p}} } }_{L^{2p}} \\
& \le C\nrm{\frac{\bfrho}{\bfc^{\dlt}}}_{L^\infty} \nrm{\bfc}_{L^\infty}^{\dlt+\gmm-1} \nrm{\bfc}_{\calX^{m+1,\gmm}}
\end{split}
\end{equation*} where we have used that \begin{equation*}
\begin{split}
\nrm{ \frac{\rd^{m-j}\bff }{ \bfc^{1-\frac{1}{p} + \frac{\gmm}{2p}} } }_{L^{2p}}  \le C\nrm{\bfc}_{\calX^{m+1,\gmm}}^{\frac{1}{p}}
\end{split}
\end{equation*} and \begin{equation*}
\begin{split}
\nrm{ \frac{\rd^{j} ( \bfc^{\gmm-1} )}{\bfc^{\gmm-1 + \frac{1}{p^*} (\frac{\gmm}{2}-1)}} }_{L^{2p^*}} \le C \nrm{\bfc}_{\calX^{m+1,\gmm}}^{\frac{1}{p^*}}. 
\end{split}
\end{equation*} This gives \begin{equation*}
\begin{split}
\left| III + \int  \rd^m\bff \cdot \rd^m \nb \bfrho \right| \le C \nrm{\frac{\bfrho}{\bfc^{\dlt}}}_{L^\infty} \nrm{\bfc}_{L^\infty}^{\dlt+\gmm-1} \nrm{\bfc}_{\calX^{m+1,\gmm}}. 
\end{split}
\end{equation*} Collecting the estimates, \begin{equation}\label{eq:calX-m-1}
\begin{split}
\left| I + II + III + IV \right| \le C  (1 + \nrm{\bfc}_{L^\infty} + \nrm{\bfrho}_{L^\infty})^b (1+\nrm{\frac{\bfc^{\dlt}}{\bfrho}}_{L^\infty} +\nrm{\frac{\bfrho}{\bfc^{\dlt}}}_{L^\infty} )(1+\nrm{\bfc}_{\calX^{m+1,\gmm}} + \nrm{\bfrho}_{\calX^{m,1}})^3. 
\end{split}
\end{equation}

\medskip

\noindent \textit{(iii) $\calX^m$--estimates, part 2}: We now need to estimate \begin{equation*}
\begin{split}
\int \frac{|\nb\bfc|^{2(m+1)}}{\bfc^{2(m+1)-2+\gmm}}, \quad \int \frac{|\nb\bfrho|^{2m}}{\bfrho^{2m-1}}. 
\end{split}
\end{equation*} We begin with writing \begin{equation*}
\begin{split}
\int \frac{|\nb\bfc|^{2(m+1)}}{\bfc^{2(m+1)-2+\gmm}} = \int \left|\frac{2(m+1)}{2-\gmm}\nb(\bfc^{\frac{2-\gmm}{2(m+1)}} )\right|^{2(m+1)} = C_{\gmm,m} \nrm{\nb (\bfc^{\alp}) }_{L^{2(m+1)}}^{2(m+1)},\quad \alp = \frac{2-\gmm}{2(m+1)}.
\end{split}
\end{equation*} Then, from \begin{equation*}
\begin{split}
\rd_t \bfc^{\alp} = -\alp \bfc^{\alp+\gmm-1} \bfrho ,
\end{split}
\end{equation*}we have \begin{equation*}
\begin{split}
\rd_t \nb (\bfc^{\alp}) = - \alp \bfc^{\gmm+\dlt-1} \frac{\bfrho}{\bfc^{\dlt}} \nb (\bfc^{\alp}) - \alp \bfc^{\alp } \nb(\bfrho \bfc^{\gmm-1}).
\end{split}
\end{equation*} We then write 
\begin{equation*}
\begin{split}
 - \alp \bfc^{\alp } \nb(\bfrho \bfc^{\gmm-1}) = -(\gmm-1) \bfc^{\gmm+\dlt-1} \frac{\bfrho}{\bfc^{\dlt}} \nb(\bfc^\alp) - \alp \bfc^{\alp+\gmm-1} \nb\bfrho 
\end{split}
\end{equation*} and furthermore \begin{equation*}
\begin{split}
\nrm{  \bfc^{\alp+\gmm-1} \nb\bfrho  }_{L^{2(m+1)}}^{2(m+1)} = \int \bfc^{2-\gmm} \bfc^{2(\gmm-1)(m+1)} \bfc^{\dlt(2m+1-\frac{1}{k})} \frac{\bfrho^{2m+1-\frac{1}{k}}}{\bfc^{\dlt(2m+1-\frac{1}{k})} } \frac{|\nb\bfrho|^{2m}}{\bfrho^{2m-1}} \frac{|\nb\bfrho|^2}{\bfrho^{2-\frac{1}{k}}}, \quad k = m- \lfloor\frac{d}{2}\rfloor - 1\ge 1. 
\end{split}
\end{equation*} We have that \begin{equation*}
\begin{split}
2-\gmm + 2(\gmm-1)(m+1) + \dlt(2m+1-\frac{1}{k}) \ge (\gmm+\dlt-1)(2m+1-\frac{1}{k}) + (2-\gmm) + (\gmm-1)(1+\frac{1}{k})
\end{split}
\end{equation*} and since $\gmm+\dlt-1\ge0$, $(2-\gmm) + (\gmm-1)(1+\frac{1}{k})\ge 0$, \begin{equation*}
\begin{split}
\nrm{  \bfc^{\alp+\gmm-1} \nb\bfrho  }_{L^{2(m+1)}}^{2(m+1)} \le C \nrm{\bfc}_{L^\infty}^{b} \nrm{\frac{\bfrho}{\bfc^\dlt}}_{L^\infty}^{2m+1-\frac{1}{k}} \nrm{\bfrho}_{\calX^{m,1}}^2.
\end{split}
\end{equation*} Returning to the equation for $\rd_t \nb (\bfc^\alp)$ and taking the $L^{2(m+1)}$-norm of both sides, \begin{equation}\label{eq:apriori4}
\begin{split}
\frac{d}{dt} \nrm{  \nb (\bfc^\alp) }_{L^{2(m+1)}} \le C(1 + \nrm{\bfc}_{L^\infty} + \nrm{\frac{\bfrho}{\bfc^\dlt}}_{L^\infty})^b (  \nrm{  \nb (\bfc^\alp) }_{L^{2(m+1)}} + \nrm{\bfrho}_{\calX^{m,1}}^{\frac{m}{m+1}} )
\end{split}
\end{equation} for some $b = b(m,\gmm,\dlt)\ge 0$. Next, with $\beta = \frac{1}{2m}$, we write
\begin{equation*}
\begin{split}
\rd_t \nb(\bfrho^\beta) = -(1+\beta) \nb(\bfrho^\beta)\lap\bfc - \nb^2 (\bfrho^\beta)\cdot\nb\bfc . 
\end{split}
\end{equation*} With an integration by parts, we estimate \begin{equation}\label{eq:apriori5}
\begin{split}
\frac{d}{dt} \nrm{\nb(\bfrho^\beta)}_{L^{2m}}\le C \nrm{\lap\bfc}_{L^\infty}\nrm{\nb(\bfrho^\beta)}_{L^{2m}}\le C\nrm{\bfc}_{\calX^{m+1,\gmm}} \nrm{\nb(\bfrho^\beta)}_{L^{2m}}. 
\end{split}
\end{equation}
Combining \eqref{eq:apriori4} and \eqref{eq:apriori5} with \eqref{eq:calX-m-1}, we obtain that \begin{equation}\label{eq:apriori-3}
\begin{split}
\frac{d}{dt} \left( \nrm{\bfc}_{\dot{\calX}^{m+1,\gmm}} + \nrm{\bfrho}_{\dot{\calX}^{m,1}} \right)  \le  C  (1 + \nrm{\bfc}_{L^\infty} + \nrm{\bfrho}_{L^\infty}+\nrm{\frac{\bfc^{\dlt}}{\bfrho}}_{L^\infty} +\nrm{\frac{\bfrho}{\bfc^{\dlt}}}_{L^\infty} )^b(1+\nrm{\bfc}_{\calX^{m+1,\gmm}} + \nrm{\bfrho}_{\calX^{m,1}})^3. 
\end{split}
\end{equation} Here, $b = b(\gmm,\dlt,m)\ge0$. In the left hand side of \eqref{eq:apriori-3}, $\dot{\calX}^{m+1,\gmm}$ and $\dot{\calX}^{m,1}$ can be replaced with ${\calX}^{m+1,\gmm}$ and ${\calX}^{m,1}$, respectively. 

\medskip

\noindent \textit{(iv) Concluding the a priori estimate}: Using \eqref{eq:apriori-1}, \eqref{eq:apriori-2}, and \eqref{eq:apriori-3}, we conclude that for \begin{equation}\label{eq:X}
\begin{split}
X :=  \nrm{\bfc}_{\dot{\calX}^{m+1,\gmm}} + \nrm{\bfrho}_{\dot{\calX}^{m,1}} + \nrm{\bfc}_{L^\infty} + \nrm{\bfrho}_{L^\infty}+\nrm{\frac{\bfc^{\dlt}}{\bfrho}}_{L^\infty} +\nrm{\frac{\bfrho}{\bfc^{\dlt}}}_{L^\infty} ,
\end{split}
\end{equation} we have the inequality \begin{equation}\label{eq:apriori}
\begin{split}
\frac{d}{dt} X \le C(1+X)^b
\end{split}
\end{equation} for some $b = b(\gmm,\dlt,m)\ge0$. {Assuming that $X$ is smooth in time, we may take $T>0$ depending only on $X(t=0)$ such that any solution to \eqref{eq:apriori} satisfies \begin{equation}\label{eq:apriori-integrated}
	\begin{split}
		\sup_{0\le t \le T} X(t) \le 2X(0). 
	\end{split}
\end{equation}}

\subsubsection{Existence and uniqueness}\label{subsubsec:eandu}

The proof of uniqueness can be done along the lines of the uniqueness proof for Theorem \ref{thm:lwp-away}, by closing an $L^2$ estimate for the difference of two hypothetical solutions. To prove existence, we consider the system \eqref{eq:KS} with initial data $\bfrho_0^{\eps} = \bfrho_0 + \eps$ and $\bfc_0^{\eps} = \bfc_0+\eps$, where $(\bfrho_0,\bfc_0)$ is the given initial data. Since $\bfrho_0^{\eps}$ and $\bfc_0^{\eps}$ are strictly positive and smooth, there is a corresponding local-in-time solution $(\bfrho^{\eps},\bfc^{\eps})$ to \eqref{eq:KS} guaranteed by Theorem \ref{thm:lwp-away}. {We can then define $X^\eps(t)$ as in \eqref{eq:X}, simply replacing $\bfc$ and $\bfrho$ with $\bfc^\eps$ and $\bfrho^\eps$; \begin{equation*}
		\begin{split}
			X^{\eps}(t) :=  \left(\nrm{\bfc^{\eps}}_{\dot{\calX}^{m+1,\gmm}} + \nrm{\bfrho^{\eps}}_{\dot{\calX}^{m,1}} + \nrm{\bfc^{\eps}}_{L^\infty} + \nrm{\bfrho^{\eps}}_{L^\infty}
			+\nrm{\frac{(\bfc^{\eps})^{\dlt}}{\bfrho^{\eps}}}_{L^\infty} +\nrm{\frac{\bfrho^{\eps}}{(\bfc^{\eps})^{\dlt}}}_{L^\infty} \right)(t).
		\end{split}
	\end{equation*} Then, it is not difficult to see that for each $\eps>0$, $X^{\eps}$ is $C^1$--in $t$ and the a priori estimate \eqref{eq:apriori} is satisfied for $X^{\eps}$ with a (possibly larger) constant $C>0$ which is uniformly bounded for $0<\eps\le1$. Furthermore, note that $X^{\eps}_0$ is uniformly bounded for $0\le \eps\le 1$ and is convergent to $X_0$ as $\eps\to 0$. Therefore, for all small $\eps>0$, there is a uniform time interval of existence $[0,T]$ with some $T>0$ for the solution  $(\bfrho^{\eps},\bfc^{\eps})$, on which we have \begin{equation*}
\begin{split}
\sup_{0\le t \le T} X^\eps(t) \le 2X^\eps(0) \le 4X(0). 
\end{split}
\end{equation*}} On the other hand, note that the quantity $X^\eps$ controls the Sobolev norm $\nrm{\bfc^{\eps}}_{H^{m+1}} + \nrm{\bfrho^{\eps}}_{H^{m}}$, and therefore we can pass to a weakly convergent subsequence $(\bfrho^{\eps},\bfc^{\eps}) \rightarrow (\bfrho,\bfc)$ in $H^{m}\times H^{m+1}$, {which in particular implies uniform pointwise convergence}. Therefore, it is straightforward to verify that the limit $(\bfrho,\bfc)$ solves \eqref{eq:KS} with the given initial data and satisfies the {estimate \eqref{eq:apriori-integrated}}. This finishes the proof. 

\subsection{Well-posedness and singularity formation with quadratic vanishing}\label{subsec:v2}

In this section, we investigate the dynamics of \eqref{eq:KSF} when the initial data vanishes \textit{quadratically} at some point in the domain. This yields another version of Theorem \ref{MainThm3}, which is stated as Theorem \ref{thm:quadratic} below. As we have explained earlier, Theorem \ref{thm:lwp-rho-away} is unable to handle initial data which vanishes with low order. Let us begin with some discussion which shows that the \eqref{eq:KSF} system could be still wellposed in that case. To avoid complications, we start by considering the simplest possible setting: assume no velocity field ($\bfu = 0$) and one-dimensional domain. Then, we have the following system 
\begin{equation}\label{eq:KS1D}
\left\{
\begin{aligned}
\rd_t \bfrho & = -\rd_x ( \bfrho \rd_x \bfc), \\
\rd_t \bfc & =   - \bfc\bfrho .
\end{aligned}
\right.
\end{equation}  We now consider the case where $\bfrho_0$ vanishes at some point $x_0$. We further assume that the second derivative of $\bfrho_0$ at $x_0$ is nonzero: that is, \begin{equation}\label{eq:van-assump-rho}
\begin{split}
\bfrho_0(x) \ge a_0(x-x_0)^2, \quad |x-x_0| < \delta_0,
\end{split}
\end{equation} for some $a_0 >0$, $x_0$, and $\dlt_0>0$. On the other hand, we keep the assumption that $\bfc_0$ is bounded away from zero. Assuming in addition that  \begin{equation}\label{eq:van-assump-c}
\begin{split}
\rd_x\bfc_0(x_0) = 0,
\end{split}
\end{equation} we have from \eqref{eq:KS1D} that the vanishing point of $\bfrho$ does not change with time. This may be viewed as the simplest scenario in which the initial data touch zero. 

\subsubsection*{Linearization against explicit quadratic solutions}
To gain some insight on whether the system \eqref{eq:KS1D} could be locally well-posed in the setup described above, we again take the linearization approach in the regime $|x|\ll1$: taking the approximate solution $(\bar{\bfc}, \bar{\bfrho}) = (1-C(t)x^2, R(t)x^2)$ with $C(0) = R(0) = 1$, writing $\bfrho = \bar{\bfrho} + \tilde{\bfrho}$, $\bfc = \bar{\bfc} + \tilde{\bfc}$, and removing terms that are quadratic in the  perturbation gives \begin{equation}  \label{eq:lin-quad-prim}
\left\{
\begin{aligned} 
&\rd_t \tilde{\bfc} = - R(t)x^2 \tilde{\bfc} - (1 - C(t)x^2)\tilde{\bfrho}, \\
&\rd_t \tilde{\bfrho} =   2C(t)\tilde{\bfrho} + 2C(t)x\rd_x\tilde{\bfrho} - 2xR(t)\rd_x\tilde{\bfc} - R(t)x^2 \rd_{xx}\tilde{\bfc}. 
\end{aligned}
\right.
\end{equation} Then, formally taking $0 < t \ll 1$ and $|x| \ll 1$, we reduce the above into  
\begin{equation}\label{eq:lin-quad}
\left\{ 
\begin{aligned}
\rd_t \tilde{\bfc} &=-\tilde{\bfrho} ,\\
\rd_t \tilde{\bfrho} &= -\rd_x (x^2\rd_x\tilde{\bfc}) . 
\end{aligned}
\right.
\end{equation} We have even removed $ 2C(t)\tilde{\bfrho} + 2C(t)x\rd_x\tilde{\bfrho}$ in the equation for $\tilde{\bfrho}$, which is not important in terms of local well-posedness. Note that for \eqref{eq:lin-quad}, \begin{equation*}
\begin{split}
\frac{1}{2}\frac{d}{dt} \left( \nrm{\tilde{\bfrho}}_{L^2}^2 + \nrm{x\rd_x\tilde{\bfc}}_{L^2}^2 \right) = 0. 
\end{split}
\end{equation*}  
This suggests local well-posedness again in this case, for $\trho, \tc \in H^\infty$. Indeed, $x$-weighted Sobolev norms control the usual Sobolev norm: assume that we have for all $m$, $\nrm{x\rd_x^m\bfc}_{L^2} \lesssim_m 1$. 
Then from \begin{equation*}
\begin{split}
\nrm{x\rd_x \bfc}_{L^2}^2 & =  -\int 2x \bfc\rd_x \bfc - \int x^2 \bfc \rd_{xx}\bfc \ge  \int \bfc^2 - \nrm{x\bfc}_{L^2}\nrm{x\rd_{xx}\bfc}_{L^2},
\end{split}
\end{equation*} we have that \begin{equation*}
\begin{split}
\nrm{\bfc}_{L^2}^2 \le \nrm{x\bfc}_{L^2}\nrm{x\rd_{xx}\bfc}_{L^2} + \nrm{x\rd_x\bfc}_{L^2}^2 ,
\end{split}
\end{equation*} and since one can replace $\bfc$ with $\rd_x^m\bfc$ throughout as well, we conclude that $\nrm{\rd_x^m\bfc}_{L^2} \lesssim_m 1$. 

\subsubsection*{Singularity formation}
Plugging in the ansatz \begin{equation}\label{eq:ode-sol}
\begin{split}
\bfc(t,x) = 1 - C(t)x^2 + O(|x|^4), \quad \bfrho(t,x) = R(t)x^2 + O(|x|^4)
\end{split}
\end{equation} with $C, R \ge 0$ to \eqref{eq:KS1D} gives the system of ODEs \begin{equation}\label{eq:ode}
\left\{ 
\begin{aligned}
\dot{C}(t) &= R(t),\\
\dot{R}(t) &= 6C(t)R(t)
\end{aligned}
\right.
\end{equation} It is clear that $C, R$ blows up in finite time once $R_0, C_0 > 0$. Indeed, for any $C^2$-initial data $(\bfc_0,\bfrho_0)$ satisfying \begin{equation*}
\begin{split}
\bfc_0(0) = 1, \quad {\rd_{xx}}\bfc_0(0) =  {-}2C_0, \quad \bfrho_0(0) = 0, \quad {\rd_{xx}}\bfrho_0(0) = 2R_0,
\end{split}
\end{equation*} it can be shown that any $C^2$ solution must blow up in finite time since \begin{equation*}
\begin{split}
C(t) := {-}\frac{1}{2}{\rd_{xx}}\bfc(t,0),\quad R(t):= \frac{1}{2}{\rd_{xx}}\bfrho(t,0)  
\end{split}
\end{equation*} solves the ODE system \eqref{eq:ode}. To rigorously conclude finite-time singularity formation, we need to establish local in time existence and uniqueness of smooth (at least $C^2$) solutions to \eqref{eq:KS1D} with initial data satisfying \eqref{eq:van-assump-rho} and \eqref{eq:van-assump-c}.

\begin{remark}
	We note that the above singularity formation result is similar to the one for the so-called Serre--Green--Naghdi equations obtained in the recent work \cite{BB}. Indeed, their result can be interpreted as a singularity formation for the (KS) system with $k(\bfc)=\bfc$ and $\chi(\bfc)=\bfc^{-1}$ (private communication with R. Granero-Belinch\'{o}n). 
\end{remark}

\subsubsection*{Local well-posedness and singularity formation}

We consider the system \eqref{eq:KS} in some $d$-dimensional domain $\Omega$. One may take either $\Omega = \bbT^k \times \bbR^{d-k}$ or some open set in $\bbR^d$ with smooth boundary. In the latter case, we do not impose any boundary conditions on $\bfrho$ and $\bfc$. 

\begin{theorem}\label{thm:quadratic}
	Assume that the initial data satisfies \begin{enumerate}
		\item (regularity) We have $\bfrho_0, \bfc_0 \in H^{\infty}(\Omega)$.
		\item (vanishing) There exists $x_0 \in \Omega$ such that \begin{equation}\label{eq:vanishing} 
		\begin{split}
		& \bfrho_0(x_0) = 0, \quad \nabla\bfrho_0(x_0) = {\bf 0}, \quad \nabla\bfc_0(x_0) = {\bf 0}.
		\end{split}
		\end{equation}
		\item (non-degeneracy) For some constants $\underline{a}_0, \underline{r}_0, \underline{c}_0, \underline{\dlt}_0 > 0$, we have \begin{equation}\label{eq:lowerbound-rho}
		\begin{split}
		\bfrho_0(x) \ge \begin{cases}
		\underline{a}_0|x-x_0|^2 , \quad & |x-x_0| < \dlt_0 \\
		\underbar{r}_0, \quad & |x-x_0| \ge \dlt_0 
		\end{cases},
		\end{split}
		\end{equation} and \begin{equation}\label{eq:lowerbound-c}
		\begin{split}
		\bfc_0(x) \ge \underline{c}_0 . 
		\end{split}
		\end{equation} 
	\end{enumerate} Then, there exist $T > 0$ and a unique solution $(\bfrho,\bfc)$ to \eqref{eq:KS} with initial data $(\bfrho_0,\bfc_0)$ satisfying $\bfrho \in C^0([0,T); H^\infty)$, $\bfc \in C^0([0,T); H^{\infty})$, the vanishing conditions \eqref{eq:vanishing} for each $0<t<T$, and the lower bounds \eqref{eq:lowerbound-rho}--\eqref{eq:lowerbound-c} with time-dependent positive constants $\underline{a}(t), \underline{r}(t), \underline{c}(t), \underline{\dlt}(t)$.

	Moreover, there exists a set of initial data satisfying the above which blows up in finite time: to be more precise, there exists some $T^*>0$ such that the unique solution blows up in $W^{2,\infty}$ \begin{equation*} 
	\begin{split}
	& {\limsup}_{t \rightarrow T^*}\left( \nrm{\bfrho(t)}_{W^{2,\infty}} + \nrm{\bfc(t)}_{W^{2,\infty}} \right) = \infty. 
	\end{split}
	\end{equation*} 
\end{theorem}

\begin{remark}
	If $\Omega$ is unbounded, we can simply modify the regularity condition as $\nabla\bfrho_0, \nabla\bfc_0 \in H^{\infty}$, $\bfrho_0, \bfc_0 \in L^\infty$. 
\end{remark}

\begin{proof}
	Without loss of generality, we shall assume that ${\bf 0} \in \Omega$ and $x_0 = {\bf 0}$. We introduce $\bff = \nb\bfc$ and write the inviscid \eqref{eq:KS} in the following slightly more convenient form: \begin{equation} \label{eq:KS-2}
	\left\{
	\begin{aligned}
	\rd_t\bfrho & = -\nb \cdot (\bfrho\bff) = - \bfrho (\nb\cdot\bff) - (\bff\cdot\nb)\bfrho, \\
	\rd_t\bff & = -\nb(\bfc\bfrho) = -\bfrho \bff - \bfc \nb\bfrho .
	\end{aligned}
	\right.
	\end{equation} For simplicity we divide the proof into a few steps.
	
	\medskip
	
	\noindent \textit{Step 1: Propagation of degeneracy}
	
	\medskip
	
	\noindent We assume \textit{formally} that there is a smooth solution ($\bfrho \in C^0_tH^m_x, \bff \in C^0_tH^{m}_x$ for some $m>\frac{d}{2}+2$ is sufficient) on some time interval $[0,T)$, and then the vanishing condition propagates in the same time interval. For this we recall \eqref{eq:vanishing} and start with \begin{equation*} 
	\begin{split}
	& \rd_t \bfrho + \bff \cdot \nb \bfrho = -\bfrho \cdot \nb \bff ;
	\end{split}
	\end{equation*}  from this it is clear that \begin{equation*} 
	\begin{split}
	& \frac{d}{dt} \bfrho(t,{\bf 0}) = 0 
	\end{split}
	\end{equation*} if $\bff(t,{\bf 0}) = {\bf 0}$. Next, \begin{equation}\label{eq:rho-deriv}
	\begin{split}
	\rd_t \rd_i \bfrho + \bff \cdot \nb \rd_i\bfrho = - \rd_i\bff \cdot \nb \bfrho  -\rd_i \bfrho \cdot \nb \bff - \bfrho\cdot \nb \rd_i\bff 
	\end{split}
	\end{equation} and we have that \begin{equation*} 
	\begin{split}
	& \frac{d}{dt} \rd_i \bfrho(t,{\bf 0}) = 0 
	\end{split}
	\end{equation*} once $\nb \bfrho(t,{\bf 0}) = {\bf 0 }$, $\bfrho(t,{\bf 0}) = 0$, and $\bff(t,\bf0) = 0$. Finally, from \begin{equation*} 
	\begin{split}
	& \rd_t \bff = - \bfrho \bff - \bfc \nb \bfrho, 
	\end{split}
	\end{equation*} we have \begin{equation*} 
	\begin{split}
	& \frac{d}{dt}\bff(t, {\bf 0}) = {\bf 0}
	\end{split}
	\end{equation*} if $\nb \bfrho(t,{\bf 0}) = {\bf 0 }$. Therefore as long as a smooth solution exists, we have \begin{equation*} 
	\begin{split}
	&	 \bfrho(t,{\bf 0}) = 0, \quad \nabla\bfrho(t,{\bf 0}) = {\bf 0}, \quad \nabla\bfc(t,{\bf 0}) = {\bf 0}.
	\end{split}
	\end{equation*}
	
	\medskip
	
	\noindent \textit{Step 2: Propagation of lower bounds}
	
	\medskip 
	
	\noindent We still assume existence of a smooth solution on some time interval $[0,T)$ and prove propagation in time of lower bounds of the form \eqref{eq:lowerbound-rho}--\eqref{eq:lowerbound-c}. We begin with $\underline{c}_0$: from the equation for $\bfc$, \begin{equation*} 
	\begin{split}
	& \frac{d}{dt} \nrm{\bfc^{-1}}_{L^\infty} \le \nrm{\bfrho}_{L^\infty} \nrm{\bfc^{-1}}_{L^\infty}
	\end{split}
	\end{equation*} and we may simply define $\underline{c}(t) $ as the solution of \begin{equation}\label{eq:def-under-c} 
	\begin{split}
	& \frac{d}{dt} \underline{c}(t)  = - \nrm{\bfrho(t)}_{L^\infty}  \underline{c}(t) , \quad \underline{c}(0) = \underline{c}_0
	\end{split}
	\end{equation} so that $\underline{c}^{-1}(t) \ge  \nrm{\bfc^{-1}}_{L^\infty} $. Next, dividing both sides of the equation for $\bfrho$ by $|x|^2$, we obtain that \begin{equation*} 
	\begin{split}
	& \rd_t \frac{\bfrho}{|x|^2} + (\bff \cdot\nb)  \frac{\bfrho}{|x|^2}  = -(\nb\cdot\bff)  \frac{\bfrho}{|x|^2}  - \frac{2\bff\cdot x}{|x|^2}  \frac{\bfrho}{|x|^2}  . 
	\end{split}
	\end{equation*} This shows that along the characteristics defined by $\bff$ we can propagate the lower bound of $ \frac{\bfrho}{|x|^2} $: define the flow $\phi$ via \begin{equation*} 
	\begin{split}
	& \frac{d}{dt}\phi(t,x) = \bff(t,\phi(t,x)) , \quad \phi(0,x) = x 
	\end{split}
	\end{equation*} and then with $R(t,x):= \frac{\bfrho(t,x)}{|x|^2}$, \begin{equation*} 
	\begin{split}
	& \frac{d}{dt} R(t,\phi(t,x)) = -\left[ \nb\cdot \bff + \frac{2\bff\cdot x}{|x|^2} \right]|_{t,\phi(t,x)} R(t,\phi(t,x))
	\end{split}
	\end{equation*} so that \begin{equation*} 
	\begin{split}
	R(t,\phi(t,x))  & = R_0(x) \exp\left(  \int_0^t -\left[ \nb\cdot \bff + \frac{2\bff\cdot x}{|x|^2} \right]|_{t',\phi(t',x)} dt' \right) \\
	& \ge \frac{\bfrho_0(x)}{|x|^2} \exp\left( {(-2-d)} \int_0^t \nrm{\nb\bff(t')}_{L^\infty} dt' \right). 
	\end{split}
	\end{equation*} Therefore, for $|x| \le \underline{\dlt}_0$, \begin{equation*} 
	\begin{split}
	& \frac{\bfrho(t,\phi(t,x))}{|\phi(t,x)|^2} \ge \underline{a}_0 \exp\left( {(-2-d)} \int_0^t \nrm{\nb\bff(t')}_{L^\infty} dt' \right). 
	\end{split}
	\end{equation*} Now, \textit{defining} $ \underline{\dlt}(t) $ via
	\begin{equation}\label{eq:def-under-dlt} 
	\begin{split}
	& \frac{d}{dt} \underline{\dlt}(t) = - \nrm{\nb\bff(t)}_{L^\infty} \underline{\dlt}(t), \quad \underline{\dlt}(0) = \underline{\dlt}_0 ,
	\end{split}
	\end{equation} we obtain that $\phi(t,B_0(\underline{\dlt}_0)) \supset B_0(\underline{\dlt}(t))$, simply because the characteristics starting on $\partial B_0(\underline{\dlt}(t))$ cannot reach $\partial B_0(\underline{\dlt}_0)$ during $[0,t]$. Hence, once we define \begin{equation}\label{eq:def-under-a} 
	\begin{split}
	&\frac{d}{dt} \underline{a}(t) = - 3\nrm{\nb\bff(t)}_{L^\infty} \underline{a}(t), \quad \underline{a}(0) = \underline{a}_0,
	\end{split}
	\end{equation} we obtain that \begin{equation*} 
	\begin{split}
	& \frac{\bfrho(t,x)}{|x|^2} \ge  \underline{a}(t) , \quad |x| \le  \underline{\dlt}(t).  
	\end{split}
	\end{equation*} Note that $\underline{\dlt}(t)$ can be recovered from $\underline{a}(t)$ and $\underline{\dlt}_0$. Hence in the following we may omit the dependence of estimates in $\underline{\dlt}(t)$. Finally, defining \begin{equation}\label{eq:def-under-r}
	\begin{split}
	\frac{d}{dt} \underline{r}(t) = - 3\nrm{\nb\bff(t)}_{L^\infty} \underline{r}(t), \quad \underline{r}(0) = \underline{r}_0,
	\end{split}
	\end{equation} we have that \begin{equation*} 
	\begin{split}
	&  {\bfrho(t,x)}  \ge  \underline{r}(t) , \quad |x| >  \underline{\dlt}(t).  
	\end{split}
	\end{equation*} We omit the proof, which is straightforward. 
	
	\medskip
	
	\noindent \textit{Step 3: A priori estimates}
	
	\medskip
	
	\noindent The following lemma allows us to identify the weight $\sqrt{\bfrho}$ as $\frac{|x|}{1+|x|}$: \begin{lemma}\label{lem:equiv}
	 {Under the hypothesis in Theorem \ref{thm:quadratic},}
		we have \begin{equation}\label{eq:equiv}
		\begin{split}
		\frac{|x|}{1+|x|} \le ( \underline{a}^{-1}(t) + \underline{r}^{-1}(t)  )^{\frac{1}{2}} \sqrt{\bfrho(t,x)} 
		\end{split}
		\end{equation} and \begin{equation}\label{eq:equiv2}
		\begin{split}
		\sqrt{\bfrho(t,x)}  \le C\nrm{\bfrho}^{ {\frac{1}{2}}}_{W^{2,\infty}} \frac{|x|}{1+|x|}. 
		\end{split}
		\end{equation} 
	\end{lemma} 
	\begin{proof}
		We assume $\dlt(t) \le 1$ and consider three regions: (i) $0 \le x < \dlt(t)$, (ii) $\dlt(t) \le x < 1$, (iii) $1 \le x$. (The proof is only simpler in the case $\dlt(t) > 1$.) The inequality is trivial for (i) from $\sqrt{\bfrho(t,x)} \ge \sqrt{\underline{a}_0}|x|$. On the other hand, for (ii) and (iii), \begin{equation*}
		\begin{split}
		\frac{|x|}{1+|x|} < 1 < (\underbar{r}(t))^{-\frac{1}{2}} \sqrt{\bfrho(t,x)} .
		\end{split}
		\end{equation*} The proof of \eqref{eq:equiv2} is trivial; just consider regions $|x| \le 1$ and $|x|>1$ separately and use Taylor's expansion in the first region. 
	\end{proof}
	
	\noindent We now fix some $m$-th order derivative $\rd^m$ with $m > d + 4$ and compute \begin{equation*} 
	\begin{split}
	& \frac{d}{dt} \left( \bfc (\rd^m \bfrho)^2 + \sum_i \bfrho (\rd^m \bff_i)^2 \right) \\
	&\quad = -\bfc \bfrho (\rd^m\bfrho)^2 - 2\bfc \rd^m\bfrho \rd^m( \nb\cdot(\bfrho\bff) ) - \sum_i \nb\cdot(\bfrho\bff) (\rd^m\bff_i)^2 - \sum_i 2\bfrho \rd^m\bff_i \rd^m( \bfrho\bff_i + \bfc\rd_i\bfrho).
	\end{split}
	\end{equation*} We then integrate both sides over $\Omega$; the first and third terms are handled as follows: \begin{equation*} 
	\begin{split}
	& \left| \int -\bfc \bfrho (\rd^m\bfrho)^2  \right| \lesssim \nrm{\bfrho}_{L^\infty} \nrm{\sqrt{\bfc} \nb^m\bfrho}_{L^2}^2,
	\end{split}
	\end{equation*} and \begin{equation*} 
	\begin{split}
	& \left| - \sum_i \nb\cdot(\bfrho\bff) (\rd^m\bff_i)^2  \right| \lesssim \nrm{\nb \bff}_{L^\infty} \nrm{\sqrt{\bfrho} \nb^m\bff}_{L^2}^2 + \nrm{\sqrt{{|(\bff\cdot\nb)\bfrho|}} \nb^m \bff}_{L^2}^2 .
	\end{split}
	\end{equation*} We further note that \begin{equation*} 
	\begin{split}
	 \nrm{\sqrt{|(\bff\cdot\nb)\bfrho|} \nb^m \bff}_{L^2}^2 &\le \nrm{\frac{1+|x|}{|x|} \bff}_{L^\infty} \nrm{\frac{1+|x|}{|x|} \nb\bfrho}_{L^\infty}  \nrm{\frac{|x|}{1+|x|} \nb^m\bff}_{L^2}^2 \\
	&\lesssim  ( \underline{a}^{-1}(t) + \underline{r}^{-1}(t)  )\nrm{\bff}_{W^{1,\infty}} \nrm{\bfrho}_{W^{2,\infty}}  \nrm{\sqrt{\bfrho} \nb^m\bff}_{L^2}^2 
	\end{split}
	\end{equation*} where we have used the pointwise inequality \eqref{eq:equiv}. 
	
	Let us now handle the other two terms; the most difficult terms combine as follows: \begin{equation*} 
	\begin{split}
	& \int -2\bfc\bfrho \rd^m\bfrho  \rd^m(\nb\cdot\bff) -\sum_i \int 2\bfc\bfrho \rd^m\bff_i \rd^m\rd_i\bfrho   =  \int 2 \nabla( \bfrho\bfc) \cdot \rd^m \bff  \rd^m \bfrho .
	\end{split}
	\end{equation*}  In absolute value, this can be bounded by \begin{equation*} 
	\begin{split}
	& \lesssim \nrm{\bff}_{L^\infty} \nrm{\bfrho}_{L^\infty}^{\frac{1}{2}} \nrm{\bfc^{-1}}_{L^\infty}^{\frac{1}{2}} \nrm{\sqrt{\bfc} \rd^m\bfrho}_{L^2} \nrm{\sqrt{\bfrho} \rd^m \bff}_{L^2} + \nrm{\bfc}_{L^\infty}^{\frac{1}{2}} \nrm{\sqrt{\bfc} \rd^m\bfrho }_{L^2} {\nrm{|\nb\bfrho| \rd^m\bff}_{L^2}. }
	\end{split}
	\end{equation*} Observing the bound \begin{equation*} 
	\begin{split}
	& |\nb\bfrho(x)| \le C\nrm{\bfrho}_{W^{2,\infty}} \frac{|x|}{1+|x|},
	\end{split}
	\end{equation*} we conclude the bound \begin{equation*} 
	\begin{split}
	& \left| \int 2 \nabla( \bfrho\bfc) \cdot \rd^m \bff  \rd^m \bfrho  \right| \lesssim   ( \underline{a}^{-1}(t) + \underline{r}^{-1}(t) + \underline{c}^{-1}(t) + 1 )(1 + \nrm{\bfc}_{W^{1,\infty}} + \nrm{\bfrho}_{W^{2,\infty}}  ) \nrm{\sqrt{\bfc} \rd^m\bfrho}_{L^2} \nrm{\sqrt{\bfrho} \rd^m \bff}_{L^2}.
	\end{split}
	\end{equation*} The remaining terms in the fourth term are straightforward to handle: for any {$\frac{m}{2} < k \le m $}, \begin{equation*}
	\begin{split}
	\left| \int 2\bfrho \rd^m\bff \rd^k\bff \rd^{m-k}\bfrho  \right|& \lesssim   \nrm{\rd^{m-k}\bfrho}_{L^\infty} \nrm{\sqrt{\bfrho}\rd^m\bff}_{L^2}\nrm{\sqrt{\bfrho}\rd^k\bff}_{L^2} \\
	&\lesssim  \nrm{ \bfrho}_{H^m} \nrm{\sqrt{\bfrho}\rd^m\bff}_{L^2}\nrm{\sqrt{\bfrho}\rd^k\bff}_{L^2},
	\end{split}
	\end{equation*} and then for {$1 \le  k \le \frac{m}{2}$}, \begin{equation*} 
	\begin{split}
	 \left| \int 2\bfrho \rd^m\bff \rd^k\bff \rd^{m-k}\bfrho  \right|& \lesssim \nrm{\bfrho}_{L^\infty} \nrm{\rd^{m-k}\bfrho}_{L^2} \nrm{\sqrt{\bfrho}\rd^m\bff}_{L^2}\nrm{\sqrt{\bfrho}\rd^k\bff}_{L^\infty} \\
	&\lesssim  \nrm{\bfrho}_{L^\infty}^{\frac{3}{2}} \nrm{\rd^{m-k}\bfrho}_{L^2} \nrm{\sqrt{\bfrho}\rd^m\bff}_{L^2}\nrm{ \bff}_{H^{m-1}} ,
	\end{split}
	\end{equation*}and the case $ k = 0$ can be bounded by \begin{equation*}
	\begin{split}
	\left| \int 2\bfrho \rd^m\bff \, \bff \rd^m\bfrho  \right| \lesssim \nrm{\bfrho}_{L^\infty}^{\frac{1}{2}} \nrm{\bff}_{L^\infty}  \nrm{\sqrt{\bfrho}\rd^m\bff}_{L^2}\nrm{\rd^m\bfrho}_{L^2}. 
	\end{split}
	\end{equation*} It only remains to consider for $1 \le  k \le m $ the following:   \begin{equation*}
	\begin{split}
	\left| \int 2\bfc \, \rd^m\bfrho\, \rd^{m+1-k}\bfrho \, \rd^{k}\bff  \right|. 
	\end{split}
	\end{equation*} When $k = m$, we can bound \begin{equation*}
	\begin{split}
	\lesssim ( \underline{a}^{-1}(t) + \underline{r}^{-1}(t) )^{\frac{1}{2}} \nrm{\bfc}_{L^\infty}^{\frac{1}{2}} \nrm{\sqrt{\bfc}\, \nb^m\bfrho}_{L^2} \nrm{\bfrho}_{W^{2,\infty}}  \nrm{\sqrt{\bfrho}\,\nb^m\bff}_{L^2} .
	\end{split}
	\end{equation*} In the opposite end when $k = 1$, we bound instead as follows: \begin{equation*}
	\begin{split}
	\lesssim \nrm{\bff}_{W^{1,\infty}} \nrm{\sqrt{\bfc}\rd_x^m\bfrho}_{L^2}^{2} .
	\end{split}
	\end{equation*} To treat the case $2 \le  k < m$, we shall prove the following elementary inequality: \begin{equation*}
	\begin{split}
	\nrm{\rd_x^{k}\bff}_{L^2} \lesssim ( \underline{a}^{-1}(t) + \underline{r}^{-1}(t) )^{\frac{1}{2}}  \left(\nrm{\sqrt{\bfrho}\rd_x^k\bff}_{L^2} + \nrm{\sqrt{\bfrho}\rd_x^{k+1}\bff}_{L^2} \right). 
	\end{split}
	\end{equation*} Appealing to Lemma \ref{lem:equiv}, it suffices to establish the following \begin{lemma}
	 {Suppose that $g\in H^1(\mathbb R^d)$. Then,}
		we have \begin{equation}\label{eq:hardy-variant} 
		\begin{split}
		& \nrm{g}_{L^2} \lesssim \nrm{\frac{|x|}{1+|x|} g}_{L^2}+ \nrm{\frac{|x|}{1+|x|} \nb g}_{L^2}.
		\end{split}
		\end{equation}
	\end{lemma}
	\begin{proof} The idea is simply to use Hardy's inequality only in the region $|x| \ll 1$. To be precise, we write \begin{equation*} 
		\begin{split}
		& \nrm{g}_{L^2}^2 = \int \frac{1}{1+ |x|^2} |g(x)|^2 + \int \frac{ |x|^2}{1+ |x|^2} |g(x)|^2 
		\end{split}
		\end{equation*} and it suffices to show that the first term on the right hand side is bounded by \eqref{eq:hardy-variant}. For this we write that \begin{equation*} 
		\begin{split}
		& \int \frac{1}{1+ |x|^2} |g(x)|^2 = \sum_i \frac{1}{d} \int \rd_i(x_i) \frac{1}{1+ |x|^2} |g(x)|^2 \\
		&=   \frac{2}{d} \int \frac{|x|^2}{(1+|x|^2)^2} |g(x)|^2 + \sum_i \frac{2}{d} \int \frac{-2x_i}{1+ |x|^2} g\rd_i g . 
		\end{split}
		\end{equation*} We consider separately the cases $d=1,2$, and {$d\ge 3$}. In the latter we have \begin{equation*} 
		\begin{split}
		& \int \left[\frac{1}{1+ |x|^2}  -  \frac{2}{d}   \frac{|x|^2}{(1+|x|^2)^2} \right]|g(x)|^2 =  \sum_i {\frac{1}{d}} \int \frac{-2x_i}{1+ |x|^2} g\rd_i g \\
		&\quad  \le {\eps} \int \frac{1}{1+|x|^2} |g(x)|^2 + C_\eps \int \frac{|x|^2}{1+|x|^2} |\nb g(x)|^2  .
		\end{split}
		\end{equation*} The proof is complete in this case since \begin{equation*} 
		\begin{split}
		& \frac{1}{1+ |x|^2}  -  \frac{2}{d}   \frac{|x|^2}{(1+|x|^2)^2}  \gtrsim \frac{1}{1+|x|^2}. 
		\end{split}
		\end{equation*} For $d = 1, 2$ we can find large $M_\eps$ depending on $\eps>0$ such that \begin{equation*} 
		\begin{split}
		& \frac{1}{1+ |x|^2}  -  \frac{2}{d}   \frac{|x|^2}{(1+|x|^2)^2} +\frac{M_\eps|x|^2}{1+|x|^2} \gtrsim \frac{2\eps}{1+|x|^2}
		\end{split}
		\end{equation*}  for all $|x|$. Then one can similarly conclude the proof. 
	\end{proof}
	
	\noindent Applying the above lemma, we obtain the bound \begin{equation*}
	\begin{split}
	&\left| \int 2\bfc \, \rd^m\bfrho\, \rd^{m+1-k}\bfrho \, \rd^{k}\bff  \right|   \lesssim \nrm{\bfc}_{L^\infty}^{\frac{1}{2}} \nrm{\sqrt{\bfc}\, \rd^m\bfrho}_{L^2} \nrm{\rd^{m+1-k}\bfrho}_{L^\infty}  \nrm{ \rd^k\bff}_{L^2}  \\
	& \lesssim ( \underline{a}^{-1}(t) + \underline{r}^{-1}(t) )^{\frac{1}{2}}    {\nrm{\bfc}_{L^\infty}} \nrm{\sqrt{\bfc}\, \nb^m\bfrho}_{L^2} \nrm{  \bfrho}_{H^m}\left(\nrm{\sqrt{\bfrho}\nb^k\bff}_{L^2} + \nrm{\sqrt{\bfrho}\nb^{k+1}\bff}_{L^2} \right)
	\end{split}
	\end{equation*}  for {$\frac{m}{2} < k < m$. In the case $k \le \frac{m}{2}$}, we can simply estimate \begin{equation*} 
	\begin{split}
	& \left| \int 2\bfc \, \rd^m\bfrho\, \rd^{m+1-k}\bfrho \, \rd^{k}\bff  \right|   \lesssim \nrm{\sqrt{\bfc}\, \nb^m\bfrho}_{L^2}\nrm{\sqrt{\bfc}\, \nb^{m+1-k}\bfrho}_{L^2} \nrm{\bff}_{H^{m-1}}. 
	\end{split}
	\end{equation*} Now we recall the weighted norms \begin{equation*} 
	\begin{split}
	& \nrm{\bfrho}_{Y^m}^2 := \sum_{j=0}^m \nrm{\sqrt{\bfc} \nb^j \bfrho}_{L^2}^2, \quad \nrm{\bff}_{Y^m}^2 := \sum_{j=0}^m \nrm{\sqrt{\bfrho} \nb^j \bff}_{L^2}^2
	\end{split}
	\end{equation*} and use \begin{equation*} 
	\begin{split}
	& \nrm{\bfrho}_{H^m} \lesssim \underline{c}^{-\frac{1}{2}}\nrm{\bfrho}_{Y^m}, \quad \nrm{\bff}_{H^{m-1}} \lesssim ( \underline{a}^{-1}(t) + \underline{r}^{-1}(t) )^{\frac{1}{2}} \nrm{\bff}_{Y^m}. 
	\end{split}
	\end{equation*} 
	
	Collecting all the estimates, \begin{equation*} 
	\begin{split}
	&\left| \frac{d}{dt} \left( \nrm{\bfrho}_{Y^m}^2 + \nrm{\bff}_{Y^m}^2 \right) \right| \lesssim \left( 1 + \underline{a}^{-1} + \underline{r}^{-1} + \underline{c}^{-1} \right)^2 \left( 1 + \nrm{\bfc}_{W^{2,\infty}} + \nrm{\bfrho}_{W^{2,\infty}} \right)^2 \left( \nrm{\bfrho}_{Y^m}^2 + \nrm{\bff}_{Y^m}^2 \right) .
	\end{split}
	\end{equation*} Strictly speaking, we have shown the above estimate for $\nrm{\bfrho}_{\dot{Y}^m}^2 + \nrm{\bff}_{\dot{Y}^m}^2$ in the left hand side but the same estimate can be shown \textit{a fortiori} for lower norms as well. We now define \begin{equation*} 
	\begin{split}
	& Z_m = \nrm{\bfrho}_{Y^m}^2 + \nrm{\bff}_{Y^m}^2 +  1 + \underline{a}^{-1} + \underline{r}^{-1} + \underline{c}^{-1}
	\end{split}
	\end{equation*} and note that \begin{equation*} 
	\begin{split}
	& \frac{d}{dt} Z_m \lesssim (Z_m)^4. 
	\end{split}
	\end{equation*} This shows that formally, there exists some time interval $[0,T)$ in which $Z_m$ remains finite. By refining the estimates suitably (see the proof in the previous section), it can be seen that this time interval does not depend on $m$ as long as $m> d + 4$. Hence on the time interval $[0,T)$, the solution belongs to $H^\infty$. Existence and uniqueness can be proved using standard arguments.

	\medskip
	
	\noindent \textit{Step 5: Blow-up criterion and finite-time singularity formation}
	
	\medskip
	
	\noindent It only remains to prove finite time singularity formation. Before we proceed, we note that \begin{equation*} 
	\begin{split}
	& \sup_{t \in [0,T)} \nrm{\bfrho(t)}_{W^{2,\infty}} + \nrm{\bfc(t)}_{W^{2,\infty}} 
	\end{split}
	\end{equation*} controls blow-up. The proof follows from applying the Gagliardo--Nirenberg--Sobolev inequalities rather than simple product estimates. We consider initial data such that near $x = {\bf 0}$,  \begin{equation*} 
	\begin{split}
	& \bfrho_0(x) = \sum_i R_{i,0} x_i^2 + O(|x|^3), \quad \bfc_0(x) = 1 - \sum_i C_{i,0} x_i^2 + O(|x|^3). 
	\end{split}
	\end{equation*} Assume towards a contradiction that the solution corresponding to the above initial data is global. We then first see that \begin{equation*} 
	\begin{split}
	& \frac{d}{dt} \rd_i\rd_j \bfrho(t,{\bf 0}) = 0 = \frac{d}{dt} \rd_i\rd_j \bfc(t,{\bf 0}) ,\quad \forall i \ne j.  
	\end{split}
	\end{equation*} Hence for all $t\ge 0$, we are guaranteed Taylor expansion of the form \begin{equation*} 
	\begin{split}
	& \bfrho(t,x) = \sum_i R_{i}(t) x_i^2 + O(|x|^3), \quad \bfc(t,x) = 1 - \sum_i C_{i}(t) x_i^2 + O(|x|^3). 
	\end{split}
	\end{equation*} {(Here, we note that the coefficient in $O(|x|^3)$ may grow in time, but it is uniformly bounded in any finite time interval.)} Then we obtain by substitution that \begin{equation}\label{eq:ODE-blowup-multiD}
	\left\{
	\begin{aligned}
	\dot{C}_i(t) & = R_i(t) \\
	\dot{R}_i(t) & = 2R_i(t)( 2C_i(t) + \sum_{j=1}^d C_j(t) ).  
	\end{aligned}
	\right.
	\end{equation} Once we choose $C_i = C > 0$ and $R_i = R > 0$ for some constants $C, R$, the system \eqref{eq:ODE-blowup-multiD} blows up in finite time, which is a contradiction. Therefore there exists some $T^*>0$ such that \begin{equation*} 
	\begin{split}
	& {\limsup}_{t \rightarrow T^*}\left( \nrm{\bfrho(t)}_{W^{2,\infty}} + \nrm{\bfc(t)}_{W^{2,\infty}} \right) = \infty. 
	\end{split}
	\end{equation*} The proof is complete. \end{proof}

\begin{remark}
	We would like to point out that the above local well-posedness and singularity formation results can be generalized in several ways. 
	
	\begin{itemize}
		\item The assumption $\nb\bfc_0(x_0)=0$ in \eqref{eq:vanishing} can be dropped. In that case, the model solution in one spatial dimension is replaced by \begin{equation*}
			\begin{split}
				\bfc(t,x) &= 1 - \ell (x-x(t)) - \frac{C(t)}{2} (x-x(t))^2 + O(|x-x(t)|^3), \\
				\bfrho(t,x) &= \frac{R(t)}{2}(x-x(t))^2 + O(|x-x(t)|^3).
			\end{split}
		\end{equation*} Here, $\ell$ is constant in time, $x(t)$ obeys the equation \begin{equation*}
			\begin{split}
				\dot{x}(t) = \ell R(t) , \quad x(0)=x_0,
			\end{split}
		\end{equation*} and $(C(t),R(t))$ still solves \eqref{eq:ode}. The local well-posedness argument goes through with $x_0$ replaced by $x(t)$. One can consider the full system involving the velocity as well; now the evolution of $x(t)$ will be affected by the velocity as well. 
		
		\item 	When the velocity becomes involved, it is an interesting problem to see whether finite time singularity formation persists: defining $x(t)$ by \begin{equation*}
			\begin{split}
				\dot{x}(t) = \bfu(t,x(t)), \quad x(0)=x_0
			\end{split}
		\end{equation*} and assuming \begin{equation*}
			\begin{split}
				\bfrho_0(x_0) = 0, \quad \nb\bfrho_0(x_0) = 0, \quad \nb\bfc_0(x_0) = 0, 
			\end{split}
		\end{equation*} the system of equations for the second derivatives $\{ \rd_{i}\rd_j \bfc(t,x(t)), \rd_i\rd_j\bfrho(t,x(t)) \}_{1\le i,j\le d}$ is given by \begin{equation}  \label{eq:ode-velocity}
			\left\{
			\begin{aligned} 
				&\frac{d}{dt}  \rd_{i}\rd_j \bfc(t,x(t))   = -\bfc_0(x_0)\rd_i\rd_j\bfrho(t,x(t)) - \rd_i\bfu(t,x(t))\cdot\nb\rd_j\bfc(t,x(t)) - \rd_j\bfu(t,x(t))\cdot\nb\rd_i\bfc(t,x(t)),\\
				&\frac{d}{dt}  \rd_{i}\rd_j \bfrho(t,x(t))   = - \rd_i\bfu(t,x(t))\cdot\nb\rd_j\bfrho(t,x(t)) - \rd_j\bfu(t,x(t))\cdot\nb\rd_i\bfrho(t,x(t)) \\
				&\quad+\sum_k\left( \rd_i\rd_k\bfrho(t,x(t))\rd_j\rd_k\bfc(t,x(t)) + \rd_j\rd_k\bfrho(t,x(t))\rd_i\rd_k\bfc(t,x(t)) + \rd_i\rd_j\bfrho(t,x(t))\rd_k\rd_k\bfc(t,x(t)) \right). 
			\end{aligned}
			\right.
		\end{equation} To close the system, we need the equation for $\nb\bfu(t,x(t))$: \begin{equation}\label{eq:ode-vel-only}
			\begin{split}
				\frac{d}{dt} \rd_i\bfu(t,x(t)) = -\rd_i\bfu(t,x(t))\cdot\nb\bfu(t,x(t)) -\rd_i\nb p(t,x(t)) + D_u \lap \rd_i\bfu(t,x(t)). 
			\end{split}
		\end{equation} Even in the inviscid case $(D_u=0)$, the system does not close due to the pressure term. Then one can try to remove the pressure term by imposing certain symmetries on the initial data. Such a symmetry assumption should be respected by the system; in particular, it should be compatible with the potential function $\phi$. In the very special case when $d = 2$ and $\phi$ is radial (i.e. function of the variable $x_1^2+x_2^2$), rotational invariance persists for the system \eqref{eq:KSF} and by taking $x(t)=x_0 = 0$, we can remove the pressure term and still conclude finite-time singularity formation. 
		
		\item The case of general $\chi, k$ can be considered. Then, the proof of local well-posedness can be shown under mild assumptions on $\chi, k$ (e.g. the one used in Theorem \ref{thm:lwp-away}) and the ODE system \eqref{eq:ode} is now given by \begin{equation}  \label{eq:ode2}
			\left\{
			\begin{aligned} 
				\dot{C}(t) &= k(1)R(t),\\
				\dot{R}(t) &= 6\chi(1)C(t)R(t). 
			\end{aligned}
			\right.
		\end{equation} 
		
		\item Although we have considered vanishing of $\bfrho$ only, one can treat the case of vanishing $\bfc$ as well {in the proof of local well-posedness. However, it is not clear to us whether finite-time singularity formation persists in this class of data.}
		
	\end{itemize}

\end{remark}

\subsection*{Acknowledgments}
IJ has been supported  by the New Faculty Startup Fund from Seoul National University, the Science Fellowship of POSCO TJ Park Foundation, and the National Research Foundation of Korea grant No. 2019R1F1A1058486. KK has been supported by NRF-2019R1A2C1084685 and NRF-2015R1A5A1009350. {We are grateful to the anonymous referees for various comments and suggestions, which have significantly improved the manuscript. Especially, we are grateful for pointing to us the possibility of the interesting generalization given in Remark \ref{rem:rotation}.}

\bibliographystyle{amsplain}
\providecommand{\bysame}{\leavevmode\hbox to3em{\hrulefill}\thinspace}
\providecommand{\MR}{\relax\ifhmode\unskip\space\fi MR }
\providecommand{\MRhref}[2]{%
	\href{http://www.ams.org/mathscinet-getitem?mr=#1}{#2}
}
\providecommand{\href}[2]{#2}


\end{document}